\documentclass[12pt]{amsart}
\usepackage{ amsmath,amssymb,amsthm,color, euscript, enumerate}
\usepackage{dsfont}

\newcommand{\1}{ \mathds{1}}

\newcommand{\Spec}{\mathrm{Spec}\ }

\newcommand{\maru}[1]{{\ooalign{\hfil#1\/\hfil\crcr
\raise.167ex\hbox{\mathhexbox20D}}}}

\newcommand{\ruby}[2]{%
 \leavevmode
 \setbox0=\hbox{#1}%
 \setbox1=\hbox{\tiny #2}%
 \ifdim\wd0>\wd1 \dimen0=\wd0 \end{lemma}se \dimen0=\wd1 \fi
 \hbox{%
   \kanjiskip=0pt plus 2fil
   \xkanjiskip=0pt plus 2fil
   \vbox{%
     \hbox to \dimen0{%
       \tiny \hfil#2\hfil}%
     \nointerlineskip
     \hbox to \dimen0{\mathstrut\hfil#1\hfil}}}}

\newcommand{\Z}{\mathbb{Z}}
\newcommand{\C}{\mathbb{C}}

\newcommand{\Q}{\mathbb{Q}}

\newcommand{\Fg}{\mathfrak{g}}

\newcommand{\Aut}{\mathrm{Aut}\,}

\newcommand{\be}{\beta}
\newcommand{\al}{\alpha}

\newcommand{\Span}{\mathrm{Span}}

\makeatletter \@addtoreset{equation}{section}

\theoremstyle{plain}
\newtheorem{theorem}{Theorem}[section]
\newtheorem{proposition}[theorem]{Proposition}
\newtheorem{lemma}[theorem]{Lemma}

\theoremstyle{definition}

\theoremstyle{remark}
\newtheorem{remark}[theorem]{Remark}
\numberwithin{equation}{section}

\title[Holomorphic vertex operator algebras]{Orbifold construction of holomorphic vertex operator algebras associated to inner automorphisms}
\author{Ching Hung Lam} %
  \address[C. H. Lam] {Institute of Mathematics, Academia Sinica, Taipei 10617, Taiwan and National Center for Theoretical Sciences of  Taiwan.}
  \email{chlam@math.sinica.edu.tw}
\author[H. Shimakura]{Hiroki Shimakura}%
\address[H. Shimakura]{Graduate School of Information Sciences,
Tohoku University,
Sendai 980-8579, Japan }%
\email {shimakura@m.tohoku.ac.jp}%
\date{}
\thanks{C.\,H. Lam was partially supported by NSC grant
  100-2628-M-001005-MY4 of Taiwan}
\thanks{H.\ Shimakura was partially supported by JSPS KAKENHI Grant Numbers 23540013 and 26800001, and by Grant for Basic Science Research Projects from The Sumitomo Foundation.}

\newcommand{\sfr}[2]{\leavevmode\kern-.1em
  \raise.5ex\hbox{\the\scriptfont0 #1}\kern-.1em
  /\kern-.15em\lower.25ex\hbox{\the\scriptfont0 #2}}

\begin{document}

\begin{abstract}
In this article, we construct three new holomorphic vertex operator algebras of central charge $24$ using the $\Z_2$-orbifold construction associated to inner automorphisms.
Their weight one subspaces has the Lie algebra structures $D_{7,3}A_{3,1}G_{2,1}$, $E_{7,3}A_{5,1}$, and $A_{8,3}A_{2,1}^2$.
In addition, we discuss the constructions of holomorphic vertex operator algebras with Lie algebras
$A_{5,6}C_{2,3}A_{1,2}$ and $D_{6,5}A_{1,1}^2$ from  holomorphic vertex operator algebras with Lie algebras
$C_{5,3}G_{2,2}A_{1,1}$ and $A_{4,5}^2$, respectively.
\end{abstract}
\maketitle


\section{Introduction}

The classification of holomorphic vertex operator algebras (VOAs) of central charge $24$ is  one of the fundamental problems in
vertex operator algebras and mathematical physics.  In 1993,  Schellekens \cite{Sc93} obtained  a partial classification by
determining possible Lie algebra structures  for the weight one subspaces  of holomorphic VOAs of  central charge $24$. There are
$71$ cases in his list but only $39$ of the $71$ cases were known explicitly at that time. It is also an open question if the Lie
algebra structure of the weight one subspace will determine the VOA structure uniquely when the central charge is $24$.

In \cite{Lam,LS,LS2},  a special class of holomorphic VOAs, called framed VOAs, of central charge $24$ were studied and classified. In particular, it was shown in \cite{LS2} that there exist exactly $56$ holomorphic framed VOAs of central charge $24$ and they are uniquely determined by the Lie algebra structures of their weight one subspaces. On the other hand, a $\Z_3$-orbifold theory associated to lattice VOAs has been developed by Miyamoto \cite{Mi3} and as an application, a holomorphic VOA whose weight one subspace has the Lie algebra structure $E_{6,3}G_{2,1}^3$ was constructed. By using the similar methods, several other holomorphic VOAs have been constructed in \cite{SS}. Recently,
van Ekeren, M\"oller and Scheithauer \cite{EMS} announced that they have obtained a mathematically rigorous proof for Schellekens' list.    They also claimed that they have established the $\Z_n$-orbifold construction for general elements of  arbitrary orders.  In particular, they claimed that they can construct two holomorphic VOAs of central charge $24$ such that their weight one subspaces have the Lie algebras structures ${E_{6,4}}{C_{2,1}}{A_{2,1}}$ and $A_{4,5}^2$.
By the results and the announcement above, there are $10$ remaining Lie algebras in Schellekens' list for which the corresponding holomorphic VOAs of central charge $24$ have not been constructed yet.

The main purpose of this article is to construct new holomorphic VOAs of central charge $24$ by using the $\Z_2$-orbifold construction associated to inner automorphisms.
More precisely, three new VOAs are constructed, and two new VOAs can be constructed from unestablished holomorphic VOAs of central charge $24$.
The main theorem is the following (see Theorems \ref{Thm:M1}, \ref{Thm:M2}, \ref{Thm:M3}, \ref{Thm:M4} and \ref{Thm:M5} for details):

\begin{theorem}\label{Thm:main}
\begin{enumerate}[{\rm (1)}]
\item There exist strongly regular, holomorphic VOAs of central charge $24$ with Lie algebras $D_{7,3}A_{3,1}G_{2,1}$, $E_{7,3}A_{5,1}$ and $A_{8,3}A_{2,1}^2$.
\item If there exists a strongly regular, holomorphic VOA of central charge $24$ with Lie algebra $C_{5,3}G_{2,2}A_{1,1}$, then there exists a strongly regular, holomorphic VOA of central charge $24$ with Lie algebra $A_{5,6}C_{2,3}A_{1,2}$.
\item If a strongly regular, holomorphic VOA of central charge $24$ with Lie algebra $A_{4,5}^2$ can be obtained by applying the $\Z_5$-orbifold construction to the lattice VOA associated to the Niemeier lattice with root lattice $A_4^6$ as in \S 10.1, then there exists a strongly regular, holomorphic VOA of central charge $24$ with Lie algebra $D_{6,5}A_{1,1}^2$.
\end{enumerate}
\end{theorem}

In order to prove this theorem, we choose a holomorphic VOA $V$ and its inner automorphism $\sigma_h$ of order $2$ carefully.
Then, applying the $\Z_2$-orbifold construction to $V$ and $\sigma_h$, we obtain a new holomorphic VOA $\tilde{V}$ with the desired Lie algebra.
We summarize the Lie algebra structures of $V_1$, $(V^{\sigma_h})_1$ and $\tilde{V}_1$ in Table \ref{T:Lie}, where $V^{\sigma_h}$ is the set of fixed-points of $\sigma_h$.

\begin{table}[bht]
\caption{Lie algebra structures of $V_1$, $(V^{\sigma_h})_1$ and $\tilde{V}_1$} \label{T:Lie}
\begin{tabular}{|c|c|c|}
\hline
(Original) Lie algebra $V_1$& (Fixed point) Lie subalgebra $(V^{\sigma_h})_1$& (New) Lie algebra $\tilde{V}_1$\\\hline\hline
$E_{6,3}G_{2,1}^3$&$D_{5,3}A_{1,1}^2A_{1,3}^2G_{2,1}U(1)$ &$D_{7,3}A_{3,1}G_{2,1}$\\\hline
$D_{7,3}A_{3,1}G_{2,1}$& $D_{6,3}A_{3,1}A_{1,1}A_{1,3}U(1)$& $E_{7,3}A_{5,1}$\\\hline
$E_{7,3}A_{5,1}$&$A_{7,3}A_{2,1}^2U(1)$ & $A_{8,3}A_{2,1}^2$\\\hline
$C_{5,3}G_{2,2}A_{1,1}$&$A_{4,6}A_{1,6}A_{1,2}U(1)^2$ &$A_{5,6}C_{2,3}A_{1,2}$\\\hline
$A_{4,5}^2$&$A_{3,5}^2U(1)^2$ &$D_{6,5}A_{1,1}^2$\\
 \hline
\end{tabular}
\end{table}

We note that van Ekeren, M\"oller and Scheithauer announced that the assumption in Theorem \ref{Thm:main} (3) is true \cite{EMS}.
We also notice that a holomorphic VOA with Lie algebra $C_{5,3}G_{2,2}A_{1,1}$ would be constructed by using $\Z_6$-orbifold theory to the lattice VOA associated to the Niemeier lattice with root lattice $E_6^4$, which will be discussed in our future article.

By the result in this article, there are remaining 5 cases in Schellekens' list which have not been constructed yet. The corresponding Lie algebras have the type $C_{4,10}$, $D_{4,12}A_{2,6}$, $A_{6,7}$, $F_{4,6}A_{2,2}$, and $C_{5,3}G_{2,2}A_{1,1}$.

\medskip

Let us explain our construction in more detail.
First, we recall the $\Z_2$-orbifold construction associated to an inner automorphism.
Let $V$ be a strongly regular, holomorphic VOA and let $h\in V_1$.
Assume that $h_{(0)}$ is semisimple on $V$ and the associated inner automorphism $\sigma_h=\exp(-2\pi\sqrt{-1}h_{(0)})$ of $V$ has order $2$.
Using Li's $\Delta$-operator introduced in \cite{Li}, we construct the (unique) irreducible $\sigma_h$-twisted $V$-module $V^{(h)}$ explicitly.
It was shown in \cite{DLM} that $V\oplus V^{(h)}$ has an abelian intertwining algebra structure.
Hence one can see that the subspace $\tilde{V}=V^{\sigma_h}\oplus (V^{(h)})_\Z$ has a VOA structure as a simple current extension of $V^{\sigma_h}$, where $V^{\sigma_h}$ is the set of fixed-points of $\sigma_h$ and $(V^{(h)})_\Z$ is the subspace of $V^{(h)}$ with integral $L(0)$-weights.
If $V$ and $h$ satisfy some conditions (see Theorem \ref{Thm:CFT} for detail), then $\tilde{V}$ is of CFT-type.
By a similar argument as in \cite{Li2}, we see that $\tilde{V}$ is $C_2$-cofinite and holomorphic.
In addition, we prove that the Lie algebras $V_1$ and $\tilde{V}_1$ share a common Cartan subalgebra under some assumptions.

Next we check that our choices of $V$ and $h$ fit the $\Z_2$-orbifold construction above.
Let $V$ be a strongly regular, holomorphic VOA of central charge $24$ such that the Lie algebra structure of $V_1$ is one of the Lie algebra structures in column one of Table \ref{T:Lie}.
Then we can easily find $h\in V_1$ so that $(V^{\sigma_h})_1$ has the Lie subalgebra structure in the corresponding column two in Table \ref{T:Lie} and $h$ satisfies the necessary conditions.
Clearly, the order of $\sigma_h$ on $V_1$ is $2$; however, we shall show that the order of $\sigma_h$ on $V$ is $2$, also.
For this purpose, we consider the subVOA  $U$ generated by $V_1$, which is a full subVOA of $V$ (\cite{DMb}).
It suffices to show that the order of $\sigma_h$ is $2$ on every irreducible $U$-submodule of $V$.
Recall that $U$ is the tensor product of simple affine VOAs $L_{\Fg_i}(k_i,0)$ associated with simple Lie algebras $\Fg_i$ at positive integral levels $k_i$ (\cite{DM06}).
Hence any irreducible $U$-module is the tensor product of irreducible $L_{\Fg_i}(k_i,0)$-modules, which are classified in \cite{FZ}.
Since the $L(0)$-weights of $V$ are integral, there are not so many possibilities for irreducible $U$-submodules of $V$.
We can check that the order of $\sigma_h$ is $2$ for each possibility.
Hence $\sigma_h$ is of order $2$ on $V$.
Unfortunately, this argument does not work for the case (3) of Theorem \ref{Thm:main}.
For this case, we directly check this assertion by using the explicit description of the lattice VOA and its irreducible $\sigma_h$-twisted module (\cite{Le,DL}).

Finally, we explain how to determine the Lie algebra structure of $\tilde{V}_1$.
A key tool is the dimension formula mentioned in \cite{Mo}.
We prove it by the following way:
We check that the character of $V^{\sigma_h}$ converges to a modular function of weight $0$ on the congruence subgroup $\Gamma_0(2)$, and we express it as a Laurent polynomial of a Hauptmodul of $\Gamma_0(2)$.
Substituting it to certain equations about the characters of $V$, $\tilde{V}$ and $V^{\sigma_h}$ and comparing some coefficients of the $q$-expansions, we obtain the formula.
By this formula, $\dim\tilde{V}_1$ can be determined by $\dim (V^{(h)})_{1/2}$.
Indeed, in all our cases, $(V^{(h)})_{1/2}=0$. It follows from the dimension formula directly or  by establishing that $(M^{(h)})_{1/2}=0$ for each possible irreducible $U$-submodule $M$ of $V$.
In the case (3) of Theorem \ref{Thm:main}, we also check it directly.
By \cite{DMb}, $\tilde{V}_1$ is semisimple, and the ratio of the dual Coxeter number and the level for every simple ideal of $\tilde{V}_1$ is determined by $\dim\tilde{V}_1$.
By using the fact that $V_1$ and $\tilde{V}_1$ share a common Cartan subalgebra $\mathfrak{H}$, any simple Lie subalgebra of $\tilde{V}_1$ spanned by weight vectors for $\mathfrak{H}$ is contained in a unique simple ideal of $\tilde{V}_1$, and the level of the simple ideal can be determined.
For example, any simple ideal of $(V^{\sigma_h})_1$ (cf.\ Table \ref{T:Lie}) is spanned by weight vectors for $\mathfrak{H}$ since $\sigma_h$ is an inner automorphism.
Thus we have enough data of $\tilde{V}_1$ to determine its Lie algebra structure in each case.

The organization of the article is as follows.
In Section 2, we recall some preliminary results
about strongly regular, holomorphic VOAs.
In Section 3, we recall the definition of Li's $\Delta$-operator and the associated construction of the $\sigma_h$-twisted module for $h\in V_1$. We also discuss some  basic properties of simple affine VOAs and their twisted modules constructed by the $\Delta$-operator.
In Section 4, we prove the dimension formula for the $\Z_2$-orbifold construction associated to an inner automorphism.
In Section 5, we discuss the $\Z_2$-orbifold construction associated to inner automorphisms.
We also discuss some properties of the resulting VOAs.
In Sections 6,7 and 8, we apply the construction successively and obtain three holomorphic VOAs of central charge $24$ with Lie algebras $D_{7,3}A_{3,1}G_{2,1}$, $E_{7,3}A_{5,1}$, and $A_{8,3}A_{2,1}^2$.
Finally, in Sections 9 and 10, we will discuss the constructions of holomorphic VOAs of central charge $24$ with Lie algebras
$A_{5,6}C_{2,3}A_{1,2}$ and $D_{6,5}A_{1,1}^2$ from holomorphic VOAs with Lie algebras
$C_{5,3}G_{2,2}A_{1,1}$ and $A_{4,5}^2$ using the similar methods, respectively.


\begin{center}
{\bf Notations}
\begin{small}
\begin{tabular}{ll}\\
$(\cdot|\cdot)$& the normalized Killing form on a semisimple Lie algebra\\
& so that $(\alpha|\alpha)=2$ for long roots $\alpha$.\\
$\langle\cdot|\cdot\rangle$& the normalized symmetric invariant bilinear form on a VOA\\
& so that $\langle \1|\1\rangle=-1$, equivalently, $\langle a|b\rangle\1=a_{(1)}b$ for $a,b\in V_1$.\\
$a_{(n)}^{(h)}$& the $n$-th mode of an element $a\in V$ on the $\sigma_h$-twisted $V$-module $M^{(h)}$.\\
$\alpha_i$& a simple root of a root system.\\
$E_{\alpha}$& a root vector in a simple Lie algebra with respect to root $\alpha$.\\
$h^\vee$ & the dual Coxeter number of a simple Lie algebra.\\
$\mathfrak{h}_{(0)}$& the subspace of fixed-points of $\tau_0$ in $\mathfrak{h}=\C\otimes_\Z N(A_4^6)$.\\
$\theta$& the highest root with respect to a fixed set of simple roots.\\
$L(0)$& the weight operator $\omega_{(1)}$.\\
$L^{(h)}(0)$& the weight operator $\omega_{(1)}^{(h)}$ on a $\sigma_h$-twisted module $M^{(h)}$.\\
$L_\mathfrak{g}(k,0)$& the simple affine VOA associated with simple Lie algebra $\mathfrak{g}$ at level $k$.\\
$L_\Fg(k,\lambda)$& the irreducible $L_\Fg(k,0)$-module with highest weight $\lambda$.\\
$\Lambda_i$& the fundamental weight with respect to simple root $\alpha_i$.\\
$M^{(h)}$& the $\sigma_h$-twisted $V$-module constructed from a $V$-module $M$ by Li's $\Delta$-operator.\\
$N=N(A_4^6)$& a Niemeier lattice with root lattice $A_4^6$.\\
$\Pi(\lambda,X_n)$& the set of all weights of the irreducible module with highest weight $\lambda$\\
& over the simple Lie algebra of type $X_n$.\\
$\rho$ & half of the sum of all positive roots.\\
$\sigma_h$& the inner automorphism $\exp(-2\pi\sqrt{-1}h_{(0)})$ of a VOA $V$ associated to $h\in V_1$.\\
$\tau_0$& the order $5$ automorphism of the lattice VOA $V_{N(A_4^6)}$ defined in \S 10.1.\\
${\rm Spec}\ h_{(0)}$& the set of spectra of $h_{(0)}$ for a semisimple element $h\in V_1$.\\
$U(1)$& a $1$-dimensional abelian Lie algebra.\\
$V^{\sigma_h}$& the set of fixed-points of $\sigma_h$, which is a full subVOA of $V$.\\
$X_n$& (the type of) a root system, a simple Lie algebra or a root lattice.\\
$X_{n,k}$& (the type of) a simple Lie algebra whose type is $X_n$ and level is $k$.\\
\end{tabular}
\end{small}
\end{center}

\section{Preliminary}
In this section, we will review some fundamental results about VOAs.

\subsection{Vertex operator algebras}
Throughout this article, all VOAs are defined over the field $\C$ of complex numbers. We recall the notion of vertex operator algebras (VOAs) and modules from \cite{Bo,FLM,FHL}.

A {\it vertex operator algebra} (VOA) $(V,Y,\1,\omega)$ is a $\Z$-graded
 vector space $V=\bigoplus_{m\in\Z}V_m$ equipped with a linear map

$$Y(a,z)=\sum_{i\in\Z}a_{(i)}z^{-i-1}\in ({\rm End}(V))[[z,z^{-1}]],\quad a\in V$$
and the {\it vacuum vector} $\1$ and the {\it conformal vector} $\omega$
satisfying a number of conditions (\cite{Bo,FLM}). We often denote it by $V$.
For $a\in V$ and $n\in\Z$, we often call $a_{(n)}$ the {\it $n$-th mode} of $a$.
Note that $L(n)=\omega_{(n+1)}$ satisfy the Virasoro relation:
$$[L_{(m)},L_{(n)}]=(m-n)L_{(m+n)}+\frac{1}{12}(m^3-m)\delta_{m+n,0}c\ {\rm id}_V,$$
where $c$ is a complex number, called the {\it central charge} of $V$.

A linear automorphism of $V$ is called {\it an automorphism} of $V$ if it satisfies $$ g\omega=\omega\quad {\rm and}\quad gY(v,z)=Y(gv,z)g\quad
\text{ for all } v\in V.$$
A {\it vertex operator subalgebra} (or a {\it subVOA}) is a graded subspace of
$V$ which has a structure of a VOA such that the operations and its grading
agree with the restriction of those of $V$ and that they share the vacuum vector.
When they also share the conformal vector, we will call it a {\it full subVOA}.
For an automorphism $g$ of a VOA $V$, let $V^g$ denote the set of fixed-points of $g$.
Clearly $V^g$ is a full subVOA of $V$.

An (ordinary) $V$-module $(M,Y_M)$ is a $\C$-graded vector space $M=\bigoplus_{m\in\C} M_{m}$ equipped with a linear map
$$Y_M(a,z)=\sum_{i\in\Z}a_{(i)}z^{-i-1}\in ({\rm End}(M))[[z,z^{-1}]],\quad a\in V$$
satisfying a number of conditions (\cite{FHL}).
We often denote it by $M$.
For an automorphism $g$ of $V$, we also consider a $g$-twisted $V$-module.
For the detail, see \cite{Li,DLM2} and references therein.
Note that a $g$-twisted $V$-module is an (untwisted) $V^g$-module.
The {\it $L(0)$-weight} of a homogeneous vector $v\in M_k$ is $k$, where $L(0)=\omega_{(1)}$.
Note that $L(0)v=kv$ if $v\in M_k$.

A VOA is said to be  {\it rational} if any module is completely reducible.
A rational VOA is said to be {\it holomorphic} if it itself is the only irreducible module up
to isomorphism.
A VOA is said to be {\it of CFT-type} if $V_0=\C\1$ (note that $V_n=0$ for all $n<0$ if $V_0=\C\1$ \cite[Lemma 5.2]{DM06b}), and is said to be {\it $C_2$-cofinite} if the codimension in $V$ of the subspace spanned by the vectors of form $u_{(-2)}v$, $u,v\in V$, is finite.
A module is said to be {\it self-dual} if its contragredient module is isomorphic to itself.
It is obvious that a holomorphic VOA is simple and self-dual.
A VOA is said to be {\it strongly regular} if it is rational, $C_2$-cofinite, self-dual and of CFT-type.

\medskip

Let $V$ be a VOA of CFT-type.
Then, the $0$-th mode gives a Lie algebra structure on $V_1$. Moreover, the $n$-th modes
$v_{(n)}$, $v\in V_1$, $n\in\Z$, define  an affine representation of the Lie algebra $V_1$ on $V$.
For a simple algebra $\mathfrak{a}$ of $V_1$, the {\it level} of $\mathfrak{a}$ is defined to be the scalar by which the canonical central element acts on $V$ as the affine representation.
When the type of the root system of $\mathfrak{a}$ is $X_n$ and the level of $\mathfrak{a}$ is $k$, we denote the type of $\mathfrak{a}$ by $X_{n,k}$.
Assume that $V$ is self-dual.
Then there exists a symmetric invariant bilinear form $\langle\cdot|\cdot\rangle$ on $V$, which is unique up to scalar (\cite{Li3}).
We normalize it so that  $\langle\1|\1\rangle=-1$.
Then for $a,b\in V_1$, we have $\langle a|b\rangle\1=a_{(1)}b$.
For an element $a\in V_1$, $\exp(a_{(0)})$ is an automorphism of $V$, called an {\it inner automorphism}.
For a semisimple element $h\in V_1$, we often consider the inner automorphism $\sigma_h=\exp(-2\pi\sqrt{-1}h_{(0)})$ associated to $h$.

Assume that $V_1$ is semisimple.
Let $\mathfrak{H}$ be a Cartan subalgebra of $V_1$.
Let $(\cdot|\cdot)$ be the Killing form on $V_1$.
We identify $\mathfrak{H}^*$ with $\mathfrak{H}$ via $(\cdot|\cdot)$ and normalize $(\cdot|\cdot)$ so that $(\alpha|\alpha)=2$ for any long root $\alpha\in\mathfrak{H}$.
In this article, {\it weights} for $\mathfrak{H}$ are defined via $(\cdot|\cdot)$, that is, the weight of a vector $v\in V$ for $\mathfrak{H}$ is $\lambda\in\mathfrak{H}$ if $x_{(0)}v=(x|\lambda)v$ for all $x\in\mathfrak{H}$.
Remark that for $h\in\mathfrak{H}$, $\sigma_h$ acts on a vector with weight $\lambda$ as the scalar multiple by $\exp(-2\pi\sqrt{-1}(h|\lambda))$.
The following lemma is immediate from the commutator relations of $n$-th modes (cf.\ {\cite[(3.2)]{DM06}}).

\begin{lemma}\label{Lem:form} If the level of a simple algebra of $V_1$ is $k$, then $\langle\cdot|\cdot\rangle=k(\cdot|\cdot)$ on it.
\end{lemma}

Let us recall some results related to the Lie algebra $V_1$.

\begin{proposition}[{\cite[Theorem 1.1, Corollary 4.3]{DM06}}]\label{Prop:posl} Let $V$ be a strongly regular, simple VOA.
Then $V_1$ is reductive.
Let $\mathfrak{s}$ be a simple Lie subalgebra of $V_1$.
Then $V$ is an integrable module for the affine representation of $\mathfrak{s}$ on $V$, and the subVOA generated by $\mathfrak{s}$ is isomorphic to the simple affine VOA associated with $\mathfrak{s}$ at positive integral level.
\end{proposition}

\begin{proposition}[{\cite[Theorem 1]{DM} and \cite[Section 3.3]{Ma}}]\label{Prop:ss} Let $V$ be a strongly regular, simple VOA and let $M$ be a $V$-module.
Then for any element $x$ in a Cartan subalgebra of $V_1$, the $0$-th mode $x_{(0)}$ acts semisimply on $M$.
\end{proposition}

\begin{proposition} [{\cite[(1.1), Theorem 3 and Proposition 4.1]{DMb}}]\label{Prop:V1} Let $V$ be a strongly regular, holomorphic VOA of central charge $24$.
If the Lie algebra $V_1$ is neither $\{0\}$ nor abelian, then $V_1$ is semisimple, and the conformal vectors of $V$ and the subVOA generated by $V_1$ are the same.
In addition, for any simple ideal of $V_1$ at level $k$, the identity$$\frac{h^\vee}{k}=\frac{\dim V_1-24}{24}$$
holds, where $h^\vee$ is the dual Coxeter number.
\end{proposition}

\section{$\Delta$-operator, simple affine VOAs and twisted modules}

In this section, we recall the twisted module constructed by Li's $\Delta$-operator.
Moreover, we discuss the lowest $L(0)$-weight of such a twisted module over simple affine VOAs.

\subsection{Twisted modules constructed by Li's $\Delta$-operator}\label{Sec:S1}
Let $V$ be a vertex operator algebra of CFT-type.
Let $\sigma$ be a finite order automorphism of $V$ and let $h\in V_1$ with $\sigma(h)=h$.
We assume that
$h_{(0)}$ acts semisimply on $V$ and that
there exists a positive integer $T\in\Z_{>0}$ such that $\Spec h_{(0)}$, the set of spectra of $h_{(0)}$ on $V$, is contained in $({1}/{T})\Z$.
Then $\sigma_h=\exp(-2\pi\sqrt{-1}h_{(0)})$ is an automorphism of $V$ with $\sigma_h^T=1$.
Note that $\sigma \sigma_h=\sigma_h \sigma$ since $\sigma(h)=h$.

Let $\Delta(h,z)$ be Li's $\Delta$-operator defined in \cite{Li}, i.e.,
\[
\Delta(h, z) = z^{h_{(0)}} \exp\left( \sum_{n=1}^\infty \frac{h_{(n)}}{-n} (-z)^{-n}\right).
\]

\begin{proposition}[{\cite[Proposition 5.4]{Li}}]\label{Prop:twist}
Let $\sigma$ be an automorphism of $V$ of finite order and
let $h\in V_1$ be as above such that $\sigma(h) = h$.
Let $(M, Y_M)$ be a $\sigma$-twisted $V$-module and
define $(M^{(h)}, Y_{M^{(h)}}(\cdot, z)) $ as follows:
\[
\begin{split}
& M^{(h)} =M \quad \text{ as a vector space;}\\
& Y_{M^{(h)}} (a, z) = Y_M(\Delta(h, z)a, z)\quad \text{ for any } a\in V.
\end{split}
\]
Then $(M^{(h)}, Y_{M^{(h)}}(\cdot, z))$ is a
$\sigma_h\sigma$-twisted $V$-module.
Furthermore, if $M$ is irreducible, then so is $M^{(h)}$.
\end{proposition}

For a $\sigma$-twisted $V$-module $M$ and $a\in V$, we denote by $a_{(i)}^{(h)}$ the operator which corresponds to the coefficient of $z^{-i-1}$ in $Y_{M^{(h)}}(a,z)$, i.e.,
$$Y_{M^{(h)}}(a,z)=\sum_{i\in\Z}a_{(i)}^{(h)}z^{-n-1}\quad \text{for}\quad a\in V.$$

Now we will review the action of some elements of $V$ on the $\sigma_h\sigma$-twisted $V$-module $M^{(h)}$.
The $0$-th mode of an element $x\in V_1$ on $M^{(h)}$ is given by
\begin{equation}
x^{(h)}_{(0)}=x_{(0)}+\langle h|x\rangle {\rm id}.\label{Eq:V1h}
\end{equation}
Let us denote by $L^{(h)}(n)$ the $(n+1)$-th mode of the conformal vector $\omega\in V$ on $M^{(h)}$.
Then the $L(0)$-weights on $M^{(h)}$ are given by
\begin{equation}
L^{(h)}(0)=L(0)+h_{(0)}+\frac{\langle h|h\rangle}{2}{\rm id}.\label{Eq:Lh}
\end{equation}
The following lemma is immediate from the equation above.

\begin{lemma}\label{Lem:wtmodule} Let $M$ be a $\sigma$-twisted $V$-module whose $L(0)$-weights are half-integral.
Let $h\in V_1$ such that $h_{(0)}$ is semisimple on $M$ and $\langle h|h\rangle\in\Z$.
Assume that the spectra of $h_{(0)}$ on $M$ are half-integral.
Then the $L(0)$-weights of the $\sigma_h\sigma$-twisted $V$-module $M^{(h)}$ are also half-integral.
\end{lemma}

\subsection{Simple affine VOAs and irreducible twisted modules}
In this subsection, we recall some properties of simple affine VOAs and their modules from \cite{Kac,FZ}.
Moreover, we study $\sigma_h$-twisted modules constructed by Li's $\Delta$-operator.

Let $\mathfrak{g}$ be a simple Lie algebra with Cartan subalgebra $\mathfrak{H}$.
Let $\Phi$ be the set of roots of $\mathfrak{g}$.
Let $(\cdot|\cdot)$ be the Killing form on $\mathfrak{g}$.
We identify $\mathfrak{H}^*$ with $\mathfrak{H}$ by $(\cdot|\cdot)$ and normalize the form so that $(\alpha|\alpha)=2$ for any long root $\alpha\in\Phi$.
Let $\{\alpha_i\mid 1\le i\le n\}\subset\mathfrak{H}$ be a set of simple roots and $\{\Lambda_i\mid 1\le i\le n\}\subset\mathfrak{H}$ the set of the fundamental weights so that $\frac{2(\Lambda_j|\alpha_i)}{(\alpha_i|\alpha_i)}=\delta_{ij}$.
For every $\be\in\Phi$, fix a root vector $E_{\beta}$ in $\mathfrak{g}$ associated to $\beta$.

Let $V=L_\Fg(k,0)$ be the simple affine VOA associated with $\mathfrak{g}$ at positive integral level $k$.
It was proved in \cite{FZ} that all irreducible $V$-modules are given by  $L_\Fg(k,\lambda)$, where $\lambda$ ranges over dominant integral weights with $(\theta|\lambda)\le k$ for the highest root $\theta$.
By \cite[Corollary 12.8]{Kac}, the lowest $L(0)$-weight of $L_\Fg(k,\lambda)$ is given by
\begin{equation}
\frac{(\lambda+2\rho|\lambda)}{2(k+{h}^\vee)},\label{Eq:kac0}
\end{equation} where $\rho=\sum_{i=1}^n\Lambda_i$ and $h^\vee$ is the dual Coxeter number.
The following facts on $L_\Fg(k,\lambda)$ is well-known.

\begin{lemma}[{\rm \cite[\S2]{FZ}}]\label{Lem:genL} Let $\ell$ be the lowest $L(0)$-weight of $L_\Fg(k,\lambda)$.
Then the following hold:
\begin{enumerate}[{\rm (1)}]
\item $L_\Fg(k,\lambda)_\ell$ is an irreducible $\mathfrak{g}$-module with highest weight $\lambda$, where $V_1\cong\mathfrak{g}$ via the $0$-th mode;
\item Let $v\in L_\Fg(k,\lambda)_\ell$.
Then the $L(0)$-weight of ${E_{\beta_1}}_{(-n_1)}\dots {E_{\beta_m}}_{(-n_m)}v$ is $\ell+\sum_{j=1}^m n_j$;
\item $L_\Fg(k,\lambda)$ is spanned by $$\{{E_{\beta_1}}_{(-n_1)}\dots {E_{\beta_m}}_{(-n_m)}v\mid \beta_i\in\Phi, n_i\in\Z_{>0}, m\in\Z_{\ge0}, v\in L_\Fg(k,\lambda)_\ell\}.$$
\end{enumerate}
\end{lemma}

From now on, let $h$ be an element in $\mathfrak{H}$ with $\Spec h_{(0)}\subset (1/T)\Z$ on $V$ for some $T\in\Z_{>0}$.
Note that $h\in \bigoplus_{i=1}^n\Q\alpha_i$ and that the restriction of the normalized Killing form $(\cdot|\cdot)$ to $ \bigoplus_{i=1}^n\Q\alpha_i$ is positive-definite.
In addition, we assume
\begin{equation}
(h|\alpha)\ge -1\qquad \text{for all } \alpha\in\Phi.\label{A:h}
\end{equation}

\begin{lemma}\label{Lem:L0weights} Let $M$ be a $V$-module.
Let $v$ be a vector in $M^{(h)}$ with $L(0)$-weight $p$, i.e., $L^{(h)}(0)v=pv$.
Let $u={E_{\beta_1}}_{(-n_1)}\dots {E_{\beta_m}}_{(-n_m)}v\in M^{(h)}$ be a non-zero vector, where $n_i\in\Z_{>0}$.
Then $u$ is a homogeneous vector in $M^{(h)}$ and its $L(0)$-weight is greater than or equal to $p$.
Moreover the equality holds if and only if $n_i=1$ and $(h|\beta_i)=-1$ for all $1\le i\le m$.
\end{lemma}
\begin{proof}
By \eqref{Eq:Lh} and $[L^{(h)}(0),{E_{\beta}}_{(-n)}]=(n+(h|\beta))id$, we have
$$L^{(h)}(0)u=\left(\sum_{i=1}^m\left( n_i+(h|\beta_i)\right)+p\right)u.$$
By the assumption \eqref{A:h}, $n_i+(h|\beta_i)\ge0$ for all $i$, and hence the $L(0)$-weight of $u$ in $M^{(h)}$ is greater than or equal to $p$.
The latter assertion is obvious.
\end{proof}

\begin{lemma}\label{Lem:cft1}
The lowest $L(0)$-weight of the irreducible $\sigma_h$-twisted $V$-module $L_\Fg(k,\lambda)^{(h)}$ is non-negative.
If the lowest $L(0)$-weight of $L_\Fg(k,\lambda)^{(h)}$ is $0$, then $\lambda=k\Lambda_j$ and $h=-\Lambda_j$ for some fundamental weight $\Lambda_j$, or $\lambda=h=0$.
\end{lemma}
\begin{proof}
Let $v$ be a lowest $L(0)$-weight vector in $L_\Fg(k,\lambda)$ with weight $\mu\in\mathfrak{H}$.
By Lemmas \ref{Lem:genL} (3) and \ref{Lem:L0weights}, we may assume that $v$ has also the lowest $L(0)$-weight in $L_\Fg(k,\lambda)^{(h)}$.
If $\lambda=0$, then $v\in\C\1$ and $\mu=0$, which proves the assertion by \eqref{Eq:Lh}.
We assume that $\lambda\neq0$.
By Lemma \ref{Lem:form}, $\langle h|h\rangle=k(h|h)$.
Hence by \eqref{Eq:Lh}, the ${L}(0)$-weight of $v$ in $L_\Fg(k,\lambda)^{(h)}$ is
\begin{equation}
\frac{(\lambda+2\rho|\lambda)}{2(k+{h}^\vee)}+(h|\mu)+\frac{k(h|h)}{2},\label{Eq:kac1}
\end{equation}
which is equal to the lowest $L(0)$-weight of $L_\Fg(k,\lambda)^{(h)}$.
Applying $2k(\lambda|\rho)\ge {h}^\vee(\lambda|\lambda)$ (\cite[Theorem 13.11]{Kac}) and $(\lambda|\lambda)\ge(\mu|\mu)$ (\cite[Proposition 11.4]{Kac}) to \eqref{Eq:kac1}, we see that the lowest $L(0)$-weight of $L_\Fg(k,\lambda)^{(h)}$ is non-negative since
\begin{equation}
\frac{(\lambda+2\rho|\lambda)}{2(k+{h}^\vee)}+(h|\mu)+\frac{k(h|h)}{2}\ge \frac{(\mu|\mu)}{2k}+(h|\mu)+\frac{k(h|h)}{2}=\frac{(\mu+kh)^2}{2k}\ge0.\label{Eq:kac2}
\end{equation}
If the $L(0)$-weight of $v$ in $L_\Fg(k,\lambda)^{(h)}$ is $0$, then the all equalities hold in \eqref{Eq:kac2}.
The former equality holds if and only if $\mu=\lambda=k\Lambda_j$ for some fundamental weight $\Lambda_j$, and the latter equality holds when $\mu+kh=0$.
Combining them, we obtain $\lambda=k\Lambda_j$ and $h=-\Lambda_j$.
\end{proof}

\begin{lemma}\label{Rem:lowestwt} Let $X_n$ be the type of the simple Lie algebra $\mathfrak{g}$.
The lowest $L(0)$-weight of $L_\Fg(k,\lambda)^{(h)}$ is equal to \begin{equation}
\frac{(\lambda+2\rho|\lambda)}{2(k+{h}^\vee)}+\min\{(h|\mu)\mid \mu\in\Pi(\lambda,X_n)\}+\frac{k(h|h)}{2},\label{Eq:twisttop}
\end{equation}
where $\Pi(\lambda,X_n)$ is the set of all weights for $\mathfrak{H}$ of the irreducible module with the highest weight $\lambda$ over the simple Lie algebra of type $X_n$.
\end{lemma}
\begin{proof} By Lemmas \ref{Lem:genL} (3) and \ref{Lem:L0weights}, it suffices to consider the lowest $L^{(h)}(0)$-weight of the lowest $L(0)$-weight space of $L_\Fg(k,\lambda)$.
Hence, this lemma follows from \eqref{Eq:kac1} and Lemma \ref{Lem:genL} (1).
\end{proof}

Later, we will use the following lemma:

\begin{lemma}\label{Lem:wtTw} Let $v$ be a vector in $L_\Fg(k,\lambda)$ with weight $\mu$ for $\mathfrak{H}$.
Then $v$ is also a weight vector in $L_\Fg(k,\lambda)^{(h)}$ for $\mathfrak{H}$ and its weight is $\mu+kh$.
\end{lemma}
\begin{proof} This lemma follows from \eqref{Eq:V1h} and Lemma \ref{Lem:form}.
\end{proof}

\subsection{Lowest $L(0)$-weight of a twisted module}
In this subsection, we give a sufficient condition so that the lowest $L(0)$-weight of the $\sigma_h$-twisted $V$-module $V^{(h)}$ is positive.

\begin{proposition}\label{Prop:cft} Let $V$ be a strongly regular, simple VOA.
Assume that the Lie algebra $\mathfrak{g}=V_1$ is semisimple.
Let $\mathfrak{g}=\bigoplus_{i=1}^t\mathfrak{g}_i$ be the decomposition into the direct sum of $t$ simple ideals $\mathfrak{g}_i$.
Let $U$ be the subVOA of $V$ generated by $V_1$.
Let $h$ be an element in a (fixed) Cartan subalgebra $\mathfrak{H}$ of $\mathfrak{g}$ such that $\Spec h_{(0)}\subset (1/T)\Z$ on $V$ for some $T\in\Z_{>0}$.
Let $h_i$ be the image of $h$ under the canonical projection from $\mathfrak{H}$ to $\mathfrak{H}\cap\mathfrak{g}_i$.
We further assume that
\begin{enumerate}[{\rm (1)}]
\item the conformal vectors of $V$ and $U$ are the same, i.e., $U$ is a full subVOA of $V$;
\item $(h|\alpha)\ge-1$ for all roots $\alpha\in\mathfrak{H}$ of $\mathfrak{g}$, where $(\cdot|\cdot)$ is the normalized Killing form on $\mathfrak{g}$ so that $(\beta|\beta)=2$ for any long root $\beta$;
\item for some $i$, $-h_i$ is not a fundamental weight.
\end{enumerate}
Then the lowest $L(0)$-weight of $V^{(h)}$ is positive.
\end{proposition}
\begin{proof} By Proposition \ref{Prop:posl}, there exist $k_i\in\Z_{>0}$ such that $U\cong\bigotimes_{i=1}^t L_{\mathfrak{g}_i}(k_i,0)$ as VOAs.
Moreover, by (1) and the rationality of $U$, $V$ is a direct sum of finitely many irreducible $U$-submodules.

Let $M$ be an irreducible $U$-submodule of $V$.
It suffices to show that the lowest $L(0)$-weight of $M^{(h)}$ is positive.
It follows from $U\cong\bigotimes_{i=1}^t L_{\mathfrak{g}_i}(k_i,0)$ that $M\cong\bigotimes_{i=1}^t  L_{\mathfrak{g}_i}(k_i,\lambda_i)$ for some dominant integral weight $\lambda_i$ of $\mathfrak{g}_i$ (cf.\ \cite{FZ}).
Let $\omega_i$ be the conformal vector of $L_{\mathfrak{g}_i}(k_i,0)$ and let $L_{\mathfrak{g}_i}(0)=(\omega_i)_{(1)}$.
Since $h=\sum_{i=1}^th_i$, we have $$L^{(h)}(0)=\sum_{i=1}^t L_{\mathfrak{g}_i}^{(h_i)}(0).$$
Clearly, (2) shows that the assumption \eqref{A:h} holds for $\mathfrak{g}_i$ and $h_i$.
Hence by Lemma \ref{Lem:cft1}
the lowest $L_{\mathfrak{g}_i}(0)$-weight of $L_{\mathfrak{g}_i}(k_i,\lambda_i)^{(h_i)}$ is non-negative, and by (3),  it is positive for at least one $i$.
Thus the lowest $L(0)$-weight of $M^{(h)}$ is positive.
\end{proof}

\section{Dimension formula associated to the $\Z_2$-orbifold construction}
In this section, we prove the dimension formula mentioned in \cite{Mo}, which will play important roles in  determining the Lie algebra structures of holomorphic VOAs.

\subsection{Characters and trace functions}

Let $U=\bigoplus_{n=0}^\infty U_n$ be a VOA of central charge $c$.
Let $f$ be an automorphism of $U$ of finite order $T$.
Let $W=\bigoplus_{n=0}^\infty W_{\lambda+n/T}$ be an irreducible $U$-module or an irreducible $f$-twisted $U$-module, where $\lambda\in\C$.
The character of $W$ is defined by the formal series $$Z_W(q)=q^{\lambda-c/24}\sum_{n=0}^\infty\dim W_n q^{n/T},$$
and the trace function of $f$ on $U$ is defined by the formal series
$$Z_U(f,q)=q^{-c/24}\sum_{n=0}^\infty {\rm Tr\ }f|_{U_n}q^{n},$$
where $q$ is a formal variable.

Now, we assume that $U$ is strongly regular and consider $Z_W(q)$ and $Z_U(f,q)$ less formally.
Take $q$ to be the usual local parameter at infinity in the upper half-plane
$$\mathbb{H}=\{\tau\in\C\mid {\rm Im}(\tau)>0\},$$
i.e, $q=e^{2\pi\sqrt{-1}\tau}$.
Since $U$ is strongly regular, $Z_{W}(q)$ and $Z_U(f,q)$ converge to holomorphic functions in $\mathbb{H}$ (\cite[Theorem 4.4.1]{Zhu} and \cite[Theorem 1.3]{DLM2}).
We often denote $Z_W(q)$ and $Z_U(f,q)$ by $Z_W(\tau)$ and $Z_U(f,\tau)$, respectively.

\subsection{Montague's dimension formula}

Let $V$ be a strongly regular, holomorphic VOA of central charge $24$.
Let $g$ be an inner automorphism of $V$ of order $2$.
Note that $g=\exp(-2\pi\sqrt{-1} h_{(0)})$ for some semisimple element $h\in V_1$.
It was proved in \cite[Theorem 1.2]{DLM2} that $V$ possesses a unique irreducible $g$-twisted $V$-module $V(g)$ up to isomorphism.

Let $S=\begin{pmatrix}0&-1\\ 1&0\end{pmatrix}$ and $T=\begin{pmatrix}1&1\\ 0&1\end{pmatrix}$ be the  standard generators of $SL(2,\Z)$.
Notice that $A=\begin{pmatrix}a&b\\ c&d\end{pmatrix}\in SL(2,\Z)$ acts on $\mathbb{H}$ by $\displaystyle A:\tau\mapsto\frac{a\tau+b}{c\tau+d}.$
By \cite[Theorem 5.3.3]{Zhu}, $Z_V(A\tau)=Z_V(\tau)$ for all $A\in SL(2,\Z)$.
Since $g$ is an inner automorphism, we can apply \cite[Theorem 1.4]{KM} to our setting and obtain
\begin{enumerate}[({\rm A}1)]
\item $Z_V(g,S\tau)=Z_{V(g)}(\tau)$.\label{Assumption2}
\end{enumerate}

In this section, we also assume the following:
\begin{enumerate}[({\rm A}1)]
\setcounter{enumi}{1}
\item $Z_{V(g)}(\tau)\in q^{-1/2}\Z[[q^{1/2}]]$, i.e., the $L(0)$-weights of $V(g)$ are positive half-integral.\label{Assumption1}
\end{enumerate}
Note that (A\ref{Assumption1}) holds if $\langle h|h\rangle\in\Z$ and the assumptions of Proposition \ref{Prop:cft} hold.

Recall that $V^g$ is a subVOA of $V$, which is the set of fixed-points of $g$.

\begin{proposition}\label{Lem:modular}
The character $Z_{V^g}(q)$ of $V^g$ converges to a holomorphic function in $\mathbb{H}$.
Moreover, it is a modular function of weight $0$ on the congruence subgroup $\Gamma_0(2)$.
\end{proposition}
\begin{proof} By definition, we obtain
\begin{equation}
Z_{V^g}(q)=\frac{1}{2}\left( Z_{V}(q)+Z_{V}(g,q)\right).\label{Eq:ZqV^+}
\end{equation}
Since both $Z_{V}(q)$ and $Z_{V}(g,q)$ converge to holomorphic functions in $\mathbb{H}$, so does $Z_{V^g}(q)$.

We denote the holomorphic function in \eqref{Eq:ZqV^+} by $Z_{V^g}(\tau)$. Next we will show that
\[
Z_{V^g}(A\tau)=Z_{V^g}(\tau)\quad \text{  for all } A\in \Gamma_0(2).
\]
It is well-known (e.g. see \cite[Theorem 4.3]{Ap}) that $\Gamma_0(2)$ is generated by $T$ and $ST^2S^{-1}$.
It follows from $Z_{V^g}(q)\in q^{-1}\Z[[q]]$ that
\begin{equation}
Z_{V^g}(T\tau)=Z_{V^g}(\tau).\label{Eq:ZVgT}
\end{equation}
In addition, by  (A\ref{Assumption1}),
\begin{equation}
Z_{V(g)}(T^2 \tau)=Z_{V(g)}(\tau).\label{Eq:ZVgT2}
\end{equation}
Notice that for any $A,B\in SL(2,\Z)$ and a meromorphic function $f(\tau)$ in $\mathbb{H}$, we have $f( AB \tau) =g(B\tau)$ where $g(\tau)= f(A\tau)$ is a meromorphic function in $\mathbb{H}$.
By (A\ref{Assumption2}), \eqref{Eq:ZVgT} and \eqref{Eq:ZVgT2} , we have
\[
Z_{V}(g, ST^2S^{-1}\tau)=Z_{V(g)}(T^2S^{-1} \tau)=Z_{V(g)}(S^{-1}\tau)=Z_{V}(g, \tau).
\]
Hence
\[
\begin{split}
Z_{V^g}(ST^2S^{-1}\tau)&=\frac{1}{2}\left( Z_{V}(ST^2S^{-1}\tau)+Z_{V}(g,ST^2S^{-1}\tau)\right)\\
&=  \frac{1}{2}\left( Z_{V}(\tau)+Z_{V}(g,\tau)\right)= Z_{V^g}(\tau)
\end{split}
\]
and
$Z_{V^g}(\tau)$ is invariant under $\Gamma_0(2)$.

In order to complete the proof, it suffices to check that $Z_{V^g}(\tau)$ is meromorphic at cusps of $\Gamma_0(2)$.
It is well-known (e.g. see \cite[Proposition 1.23]{Ha}) that $\Gamma_0(2)$ has two cusps, represented by $i\infty$ and $0$.
It is clear that $Z_{V^g}(\tau)$ is meromorphic at $i\infty$ since $Z_{V^g}(q)\in q^{-1}\Z[[q]]$.
By (A\ref{Assumption2}) and \eqref{Eq:ZqV^+}, we have $$Z_{V^g}(S\tau)=\frac{1}{2}\left( Z_V(\tau)+Z_{V(g)}(\tau)\right)$$ and this function belongs to $q^{-1}\Z[[q^{1/2}]]$ by (A\ref{Assumption1}).
Since $S$ sends $i\infty$ to $0$, the function $Z_{V^g}(\tau)$ is meromorphic at $0$.
\end{proof}

Let $V(g)_\Z$ be the subspace of $V(g)$ with integral $L(0)$-weights and set $\tilde{V}=V^{g}\oplus V(g)_\Z$.
We now assume that $\tilde{V}$ has a strongly regular, holomorphic VOA structure.
Note that
\begin{equation}
Z_{\tilde{V}}(\tau)=Z_{V^g}(\tau)+Z_{V(g)_\Z}(\tau).\label{Eq:ZtV}
\end{equation}
First, we prove the following equation, which was already mentioned in \cite[(8)]{Mo}:

\begin{proposition}\label{Prop:eqZV}
$$ Z_V(\tau)+Z_{\tilde{V}}(\tau)=Z_{V^g}(\tau)+Z_{V^g}(S\tau)+Z_{V^g}(ST\tau).$$
\end{proposition}
\begin{proof} By definition,
\begin{equation}
Z_{V^g}(\tau)=\frac{1}{2}\left( Z_{V}(\tau)+Z_V(g,\tau)\right).\label{Eq:ZV^+}
\end{equation}
Since $V$ is holomorphic, $Z_V(\tau)$ is invariant under  the action of $SL(2,\Z)$.
Hence, by (A\ref{Assumption2}), we obtain
\begin{equation}
Z_{V^g}(S\tau)=\frac{1}{2}\left( Z_{V}(\tau)+Z_{V(g)}(\tau)\right).\label{Eq:ZV^+S}
\end{equation}
It follows from $Z_V(T\tau)=Z_V(\tau)$ that
\begin{equation}
Z_{V^g}(ST\tau)=\frac{1}{2}\left( Z_{V}(\tau)+Z_{V(g)}(T\tau)\right).\label{Eq:ZV^+TS}
\end{equation}
On the other hand, by (A\ref{Assumption1}), we have
\begin{equation}
Z_{V(g)_\Z}(\tau)=\frac{1}{2}\left( Z_{V(g)}(\tau)+Z_{V(g)}(T\tau)\right).\label{Eq:ZVgZ}
\end{equation}
Thus by \eqref{Eq:ZtV}, \eqref{Eq:ZV^+S}, \eqref{Eq:ZV^+TS} and \eqref{Eq:ZVgZ}, we have
\begin{align*}
Z_{V^g}(\tau)+Z_{V^g}(S\tau)+Z_{V^g}(ST\tau)&=Z_{V^g}(\tau)+Z_V(\tau)+\frac{1}{2}\left(Z_{V(g)}(\tau)+Z_{V(g)}(T\tau)\right)\\
&=Z_V(\tau)+Z_{\tilde{V}}(\tau).
\end{align*}
\end{proof}


We now prove the following dimension formula described in \cite[(10),(11)]{Mo}.

\begin{theorem}\label{Thm:Dimformula} The following equations hold:
\begin{enumerate}[{\rm (1)}]
\item $\displaystyle\dim V(g)_{1/2}=\frac{\dim (V^g)_2-98580}{2^{11}};$
\item $\dim V_1+\dim \tilde{V}_1=3\dim (V^g)_1+24(1-\dim V(g)_{1/2}).$
\end{enumerate}
\end{theorem}
\begin{proof} It is well-known (e.g. see \cite[Exercise 3.18]{Ha}) that a hauptmodul for $\Gamma_0(2)$ is given by $$f(\tau):=\frac{\eta(\tau)^{24}}{\eta(2\tau)^{24}}=q^{-1}-24+\binom{24}{2}q+\cdots,$$ where $\eta(\tau)=q^{1/24}\prod_{n=1}^\infty(1-q^n)$, $q=e^{2\pi\sqrt{-1}\tau}$, is the Dedekind eta function.
By Lemma \ref{Lem:modular}, $Z_{V^g}(\tau)$ is a modular function of weight $0$ on $\Gamma_0(2)$ and holomorphic in $\mathbb{H}$.
In particular, $Z_{V^g}(\tau)$ is a rational function of $f(\tau)$.
In addition, since the set of all cusps of $\Gamma_0(2)$ is $\{0, i\infty\}$ and that $f(\tau)\to0$ as $\tau\to0$ (see \eqref{Eq:fS} below), $Z_{V^g}(\tau)$ is a Laurent polynomial of $f(\tau)$, i.e., $$Z_{V^g}(\tau)=\sum_{n\in\Z} c_n f^n(\tau),$$
where only finitely many coefficients $c_n\in\C$ are non-zero.
It follows from $Z_{V^g}(\tau)\in q^{-1}+\Z[[q]]$ that $c_n=0$ if $n>1$, and $c_1=1$.
Since $\eta(S\tau)=(-i\tau)^{1/2}\eta(\tau)$, we have
\begin{align}f(S\tau)&={2^{12}}\frac{\eta(\tau)^{24}}{\eta(\tau/2)^{24}}=2^{12}\left(q^{1/2}+24q+\cdots\right),\label{Eq:fS}\\
f^{-1}(S\tau)&=\frac{1}{2^{12}}\frac{\eta(\tau/2)^{24}}{\eta(\tau)^{24}}=\frac{1}{2^{12}}\left(q^{-1/2}-24+\binom{24}{2}q^{1/2}+\cdots\right)\label{Eq:fS2},\\
f^{-2}(S\tau)&=\frac{1}{2^{24}}\frac{\eta(\tau/2)^{48}}{\eta(\tau)^{48}}=\frac{1}{2^{24}}\left(q^{-1}-48q^{-1/2}+\binom{48}{2}+\cdots\right),\label{Eq:fS3}
\end{align}
which shows that $f^{n}(S\tau)\in q^{n/2}\Q[[q^{1/2}]]$.
By \eqref{Eq:ZV^+S} and the assumption (A\ref{Assumption1}), we have $Z_{V^g}(S\tau)\in \frac{1}{2}q^{-1}+q^{-1/2}\Z[[q^{1/2}]]$.
Hence $c_n=0$ if $n<-2$, and $c_{-2}=2^{23}$.
Thus
\begin{equation}
Z_{V^g}(\tau)=f(\tau)+c_0+c_{-1}f^{-1}(\tau)+2^{23}f^{-2}(\tau)=q^{-1}+\dim (V^g)_1+\cdots.\label{Eq:ZV+}
\end{equation}
Comparing the constant terms of the equation above, we have
\begin{equation}
\dim (V^g)_1=c_0-24.\label{Eq:c0}
\end{equation}

By \eqref{Eq:fS}, \eqref{Eq:fS2}, \eqref{Eq:fS3} and \eqref{Eq:ZV+}, we have
$$Z_{V^g}(S\tau)=\frac{1}{2}q^{-1}+\left(\frac{c_{-1}}{2^{12}}-24\right)q^{-1/2}+\left(c_0-\frac{24c_{-1}}{2^{12}}+\frac{1}{2}\binom{48}{2}\right)+\cdots.$$
Hence, comparing the coefficients of $q^{-1/2}$ in \eqref{Eq:ZV^+S}, we have
\begin{equation}
\frac{c_{-1}}{2^{12}}-24=\frac{1}{2}\dim V(g)_{1/2}.\label{Eq:c-1}
\end{equation}
Combining \eqref{Eq:ZV+} and \eqref{Eq:c-1} and comparing the coefficients of $q$, we obtain $$\dim (V^g)_2=\binom{24}{2}+2^{11}\dim V(g)_{1/2}+24\times2^{12},$$
which is equivalent to the equation (1).

Note that $T(q^{n/2})=(-1)^{n}q^{n/2}$ for $n\in \Z$.
Comparing the constant terms of both sides of the equation in Proposition \ref{Prop:eqZV}, we obtain
\begin{equation*}
\dim V_1+\dim \tilde{V}_1=\dim (V^g)_1+2\left(c_0-\frac{24c_{-1}}{2^{12}}+\frac{1}{2}\binom{48}{2}\right).
\end{equation*}
Hence by \eqref{Eq:c0} and \eqref{Eq:c-1}, we obtain the equation (2).
\end{proof}

\section{$\Z_2$-orbifold construction associated to inner automorphisms}\label{Sec:Z2}
In this section, we establish the $\Z_2$-orbifold construction of holomorphic VOAs associated to inner automorphisms based on \cite{DLM}.

Let us recall the setting of \cite[\S 3]{DLM}.
Let $V$ be a strongly regular, holomorphic VOA.
Then $V_1$ is a reductive Lie algebra by Proposition \ref{Prop:posl}.
Let $h\in V_1$ such that $\langle h|h\rangle\in\Z$ and $h_{(0)}$ is semisimple on $V$.
We further assume that $\Spec h_{(0)}\subset \Z/2$ on $V$ but $\Spec h_{(0)}\not\subset \Z$ on $V$.
Then $\sigma_h=\exp(-2\pi\sqrt{-1} h_{(0)})\in \Aut V$ is of order $2$.

Set $L=\Z h$.
For $r\in\Z$ and $\nu\in\C h$, we set $$V^{(rh,\nu)}:=\{v\in V^{(rh)}\mid h_{(0)}^{(rh)}v=\langle h|rh+\nu\rangle v\},$$
and we set $P=\{\nu\in\C h\mid V^{(0,\nu)}\neq0\}$.
Since $V$ is simple, there exists $s\in \Q h$ such that $P=\Z s$.
It follows from $|\sigma_h|=2$ that $\langle h|s\rangle\in\Z+1/2$.
For $i,j\in\{0,1\}$, we set
\begin{align*}
U^{(ih, js)}:=\bigoplus_{\beta\in 2P+js}V^{(ih,\beta)}.
\end{align*}
Note that
\begin{equation}
\begin{split}
&U^{(0,0)}=V^{\sigma_h}=\{v\in V\mid \sigma_h(v)= v\},\qquad U^{(h,0)}= (U^{(0,0)})^{(h)},\\
&U^{(0,s)}=\{v\in V\mid \sigma_h(v)=-v\},\qquad\qquad U^{(h,s)}=(U^{(0,s)})^{(h)}
\end{split}\label{Eq:U00}
\end{equation}
and that
\begin{align*}
V=U^{(0,0)}\oplus U^{(0,s)},\qquad V^{(h)}=U^{(h,0)}\oplus U^{(h,s)}.
\end{align*}
Let $(V^{(h)})_\Z$ be the subspace of $V^{(h)}$ with integral $L(0)$-weights.
By \eqref{Eq:Lh}, we have
\begin{align}
(V^{(h)})_\Z=\bigoplus_{n\in\Z} (V^{(h)})_n=
\begin{cases}U^{(h,0)} &{\rm if}\ \langle h|h\rangle\in 2\Z,\\
U^{(h,s)}&{\rm if}\ \langle h|h\rangle\in 2\Z+1.
\end{cases}\label{Eq:Vh}
\end{align}
Since $V$ is strongly regular, so is $U^{(0,0)}(=V^{\sigma_h})$ by \cite{Mi4,Mi}.
In addition, by \cite{DMq}, $U^{(0,0)}$ is simple, and $U^{(0,s)}$ is an irreducible $U^{(0,0)}$-module.
For $\nu\in \{0,s\}$ it follows from $(U^{(0,\nu)})^{(h)}=U^{(h,\nu)}$ and Proposition \ref{Prop:twist} that $U^{(h,\nu)}$ is an irreducible $U^{(0,0)}$-module.
The following classification of irreducible $U^{(0,0)}$-modules is a consequence of \cite{Mi4,Mi}:

\begin{proposition}\label{Prop:modU00} There exist exactly $4$ non-isomorphic irreducible $U^{(0,0)}$-modules $U^{(ih,js)}$, $i,j\in\{0,1\}$.
\end{proposition}

\subsection{Simple current modules}

In this subsection, we prove that irreducible $U^{(0,0)}$-modules, $U^{(ih,js)}$, $i,j\in\{0,1\}$, are simple current modules using the Verlinde formula proved in \cite{Huang}.
By \eqref{Eq:ZV^+S}, we have
$$Z_{U^{(0,0)}}(S\tau)=\frac{1}{2}\left(Z_{U^{(0,0)}}(\tau)+Z_{U^{(0,s)}}(\tau)+Z_{U^{(h,0)}}(\tau)+Z_{U^{(h,s)}}(\tau)\right).$$
Since $Z_{U^{(0,s)}}(\tau)=Z_{V}(\tau)-Z_{U^{(0,0)}}(\tau)$, we have
$$Z_{U^{(0,s)}}(S\tau)=\frac{1}{2}\left(Z_{U^{(0,0)}}(\tau)+Z_{U^{(0,s)}}(\tau)-Z_{U^{(h,0)}}(\tau)-Z_{U^{(h,s)}}(\tau)\right).$$

Let ${S}=(S_{P,Q})$ be the $S$-matrix indexed by $U^{(0,0)},U^{(0,s)},U^{(h,0)},U^{(h,s)}$.
Note that $S$ is symmetric and $S^2$ is the permutation matrix which sends an irreducible $U^{(0,0)}$-module to its contragredient module.
By the $S$-transformations above , we have
$$S=\frac{1}{2}\begin{pmatrix}1&1&1&1\\ 1&1&-1&-1\\ 1&-1&a&-a\\ 1&-1&-a&a\end{pmatrix},$$
where $a=1$ or $-1$.
In particular, $S^2$ is the identity matrix, which shows that any irreducible $U^{(0,0)}$-module is self-dual.
By the Verlinde formula, we have the fusion rules $$N_{P,P}^Q=\sum_{i,j\in\{0,1\}}\frac{S_{PU^{(ih,js)}}^2S_{U^{(ih,js)}Q}}{S_{U^{(0,0)}U^{(ih,js)}}}=\sum_{i,j\in\{0,1\}}S_{U^{(ih,js)}Q}.$$
Hence $N_{P,P}^Q\neq0$ if and only if $Q=U^{(0,0)}$, and $N_{P,P}^{U^{(0,0)}}=1$ for  $P=U^{(ih,js)}$, $i,j\in\{0,1\}$.
Thus we have the following proposition:

\begin{proposition}\label{Prop:SC} For $i,j\in\{0,1\}$, $U^{(ih,is)}$ is a self-dual simple current $U^{(0,0)}$-module.
\end{proposition}

\subsection{$\Z_2$-orbifold construction associated to inner automorphisms}\label{Sec:2.2}
In this subsection, we prove the following proposition by using the results in \cite[Section 3]{DLM}.

\begin{proposition}\label{Prop:VOAstr}
The $V^{\sigma_h}$-module $\tilde{V}=V^{\sigma_h}\oplus (V^{(h)})_\Z$ has a VOA structure as a simple current extension of $V^{\sigma_h}$ graded by $\Z_2$.
\end{proposition}
\begin{proof}
By \cite[Theorem 3.21]{DLM} $\bar{U}=U^{(0,0)}\oplus U^{(h,0)}\oplus U^{(0,s)}\oplus U^{(h,s)}$ is an abelian intertwining algebra.
For the notations $U^{(ih,js)}$ and $(V^{(h)})_\Z$, see \eqref{Eq:U00} and \eqref{Eq:Vh}.
Since the subspace $\tilde{V}$ of $\bar{U}$ is a $V^{\sigma_h}$-module, it satisfies the axiom of a VOA except for the Jacobi identity on $(V^{(h)})_\Z$.
Note that the generalized Jacobi identity in \cite[(3.88)]{DLM} on $\bar{U}$ is
\begin{equation}
\begin{split}
 &z_{0}^{-1}\delta\left(\frac{z_{1}-z_{2}}{z_{0}}\right)
\left(\frac{z_{1}-z_{2}}{z_{0}}\right)^{\eta((\lambda_i,\alpha),(\lambda_j,\beta))}
\bar{Y}(u,z_{1})\bar{Y}(v,z_{2})w \\
&-\bar{C}((\lambda_i,\alpha),(\lambda_j,\beta))z_{0}^{-1}\delta\left(\frac{z_{2}-z_{1}}
{-z_{0}}\right)
\left(\frac{z_{2}-z_{1}}{z_{0}}\right)^{\eta((\lambda_i,\alpha),(\lambda_j,\beta))}
\bar{Y}(v,z_{2})\bar{Y}(u,z_{1})w \\
=&z_{2}^{-1}\delta\left(\frac{z_{1}-z_{0}}{z_{2}}\right)
\left(\frac{z_{2}+z_{0}}{z_{1}}\right)^{\eta((\lambda_i,\alpha),(\lambda_k,\gamma))}
h(i,j,k)\bar{Y}(\bar{Y}(u,z_{0})v,z_{2})w,\label{Eq:GJa}
\end{split}
\end{equation}
where $u\in U^{(\lambda_i,\alpha)}$, $v\in U^{(\lambda_j,\beta)}$, $w\in U^{(\lambda_k,\gamma)}$.
We will check that the maps $\eta(\cdot,\cdot)$, $\bar{C}(\cdot,\cdot)$, $h(\cdot,\cdot,\cdot)$ are trivial on $(V^{(h)})_\Z$, equivalently, \eqref{Eq:GJa} is the usual Jacobi identity of a VOA on $(V^{(h)})_\Z$.

First, we consider the map $\eta:(L/2L\times P/2P)\times(L/2L\times P/2P)\to (\Z/2)/2\Z$ defined in \cite[(3.19)]{DLM}, where
\begin{align*}
\eta((\lambda_i,\alpha),(\lambda_j,\beta))=-\langle\lambda_i|\lambda_j\rangle-\langle\lambda_i|\beta\rangle-\langle\lambda_j|\alpha\rangle+2\Z.
\end{align*}
It is obvious that $\eta\equiv0$ on $\{(h,0)\}$ (resp. $\{(h,s)\}$) if $\langle h|h\rangle\in2\Z$ (resp. $\langle h|h\rangle\in2\Z+1$).

Next we consider the map $\bar{C}:(L\times P)\times(L\times P)\to \C^*$ defined in \cite[(3.20) and (3.87)]{DLM}, where $$\bar{C}((\lambda_i,\alpha),(\lambda_j,\beta))=e^{(\langle\lambda_i|\beta\rangle-\langle\lambda_j|\alpha\rangle)\pi\sqrt{-1}}C_1(\alpha,\beta).$$
Notice that by \cite[Remark 3.16]{DLM}, $C_1\equiv 1$ on $P$ since the rank of $P$ is one.
Hence $\bar{C}\equiv1$ on both $\{(h,0)\}$ and $\{(h,s)\}$.

Finally, we consider the map $h:\bar{A}\times\bar{A}\times\bar{A}\to \C^*$ defined in \cite[(3.83)]{DLM}, where
$$ h((\lambda_i,\alpha_1),(\lambda_j,\alpha_2),(\lambda_k,\alpha_3))=e^{-(\lambda_i+\lambda_j-\lambda_{i+j},\lambda_k)\pi\sqrt{-1}}C_1(\lambda_i+\lambda_j-\lambda_{i+j},\lambda_k)^2,$$
$\bar{A}=(L\times P)/\{(\alpha,-\alpha)\mid \alpha\in 2L\}$ and $\lambda_{i+j}\in\{0,h\}$ so that $\lambda_{i+j}\equiv \lambda_i+\lambda_j\pmod{2L}$.
Notice that $\bar{A}$ is identified with $\{(0,\alpha),(h,\alpha)\mid \alpha\in P\}$ and $\bar{U}$ is isomorphic to $\bigoplus_{(\lambda,\alpha)\in \bar{A}}V^{(\lambda,\alpha)}$ as a vector space.
In our case, $\lambda_i+\lambda_j-\lambda_{i+j}$ must be $0$ or $2h$, which shows that $e^{-(\lambda_i+\lambda_j-\lambda_{i+j},\lambda_k)\pi\sqrt{-1}}\equiv 1$.
As we mentioned above, $C_1\equiv1$ on $P$.
Hence $h(\cdot,\cdot,\cdot)\equiv 1$ on $\bar{A}\times\bar{A}\times\bar{A}$.
Thus \eqref{Eq:GJa} is the usual Jacobi identity on $(V^{(h)})_\Z$.

By Proposition \ref{Prop:SC}, $\tilde{V}$ is a simple current extension of $V^{\sigma_h}$ graded by $\Z_2$.
\end{proof}

\subsection{Properties of the VOAs obtained by $\Z_2$-orbifold construction}
In this subsection, under some assumptions, we show that the VOA $\tilde{V}$ constructed in Proposition \ref{Prop:VOAstr} is holomorphic and strongly regular and the Lie algebras $V_1$ and $\tilde{V}_1$ share a common Cartan subalgebra.
In particular, the Lie ranks of $V_1$ and $\tilde{V}_1$ are the same.

Recall that $h$ is a semisimple element in $V_1$, $\sigma_h\in\Aut V$ is of order $2$ on $V$ and $\langle h|h\rangle\in\Z$.

\begin{theorem}\label{Thm:CFT} Assume that $V_1$ is semisimple.
Let $V_1=\bigoplus_{i=1}^t\mathfrak{g}_i$ be the decomposition of $V_1$ into the direct sum of $t$ simple ideals $\mathfrak{g}_i$.
Let $k_i$ be the level of the affine representation of $\mathfrak{g}_i$ on $V$.
Let $\mathfrak{H}$ be a Cartan subalgebra of $V_1$ such that $h\in\mathfrak{H}$.
\begin{enumerate}[{\rm (1)}]
\item Assume the following:
\begin{enumerate}[{\rm (a)}]
\item the conformal vectors of $V$ and the subVOA generated by $V_1$ are the same;
\item $(h|\alpha)\ge-1$ for all roots $\alpha$ of $V_1$;
\item for some $i$, $-h_i$ is not a fundamental weight for $\mathfrak{H}$ on $V$, where $h_i$ is the image of $h$ under the canonical projection from $\mathfrak{H}$ to $\mathfrak{H}\cap\mathfrak{g}_i$.
\end{enumerate}
Then the VOA $\tilde{V}$ is of CFT-type.
\item If $\tilde{V}$ is of CFT-type, then it is strongly regular and holomorphic.
\item Assume that $\tilde{V}$ is strongly regular and holomorphic and that
\begin{enumerate}[{\rm (d)}]
\setcounter{enumi}{3}
\item $-\sum_{i=1}^tk_ih_i$ is not a weight of $V$ for $\mathfrak{H}$,
\end{enumerate}
then  $\mathfrak{H}$ is a Cartan subalgebra of $\tilde{V}_1$.
In particular, the Lie ranks of $V_1$ and $\tilde{V}_1$ are the same.
\end{enumerate}
\end{theorem}
\begin{proof} (1) By Proposition \ref{Prop:cft}, along with the assumptions (a), (b) and (c), the lowest $L(0)$-weight of $({V}^{(h)})_\Z$ is greater than $0$.
Since $V^{\sigma_h}$ is of CFT-type, so is $\tilde{V}=V^{\sigma_h}\oplus ({V}^{(h)})_\Z$.

(2) Since $V^{\sigma_h}$ is rational and $C_2$-cofinite (\cite{Mi4,Mi}), so is $\tilde{V}$ \cite[Theorem 5.6]{Li2}.
Hence $\tilde{V}$ is strongly regular.
Moreover, by Proposition \ref{Prop:modU00} and the classification of irreducible $\tilde{V}$-modules (\cite[Theorem 5.6]{Li2}), $\tilde{V}$ is holomorphic.

(3) By Proposition \ref{Prop:posl}, $\tilde{V}_1$ is reductive.
Since $\sigma_h$ acts trivially on $\mathfrak{H}$, $(V^{\sigma_h})_1$ contains $\mathfrak{H}$.
By Proposition \ref{Prop:ss}, any element of $\mathfrak{H}$ is semisimple in $\tilde{V}_1$.
Hence $\mathfrak{H}$ is a toral subalgebra of $\tilde{V}_1$.
Let us show that the centralizer $\mathfrak{z}$ of $\mathfrak{H}$ in $\tilde{V}_1$ is $\mathfrak{H}$ itself.
Clearly, $\mathfrak{H}\subset\mathfrak{z}$.
Let $v=v_1+v_2\in\mathfrak{z}$, where $v_1\in (V^{\sigma_h})_1$ and $v_2\in (V^{(h)})_1$.
Since $\mathfrak{H}\subset(V^{\sigma_h})_1$ and $(V^{(h)})_1$ is a $(V^{\sigma_h})_1$-module, both $v_1$ and $v_2$ belong to $\mathfrak{z}$.
If $v_2\neq0$, then, by Lemma \ref{Lem:wtTw}, $v_2$ is also a weight vector in $V$ for $\mathfrak{H}$ and its weight is $-\sum_{i=1}^t k_ih_i$, which contradicts (d).
Hence $v_2=0$.
Since $\mathfrak{H}$ is a Cartan subalgebra of $V_1$ and $v=v_1\in (V^{\sigma_h})_1\subset V_1$, we have $v\in \mathfrak{H}$, and hence $\mathfrak{z}\subset\mathfrak{H}$.
Thus $\mathfrak{z}=\mathfrak{H}$, and $\mathfrak{H}$ is a Cartan subalgebra of $\tilde{V}_1$.
\end{proof}

The next proposition will be used to identify the Lie algebra structure of $\tilde{V}_1$.

\begin{proposition}\label{Prop:levels} Let $V$ be a strongly regular, simple VOA.
Let $\mathfrak{H}$ be a Cartan subalgebra of the reductive Lie algebra $V_1$.
Let $\mathfrak{s}$ be a simple Lie subalgebra of $V_1$ of type $X_{n,k}$.
Assume that $\mathfrak{s}$ is spanned by weight vectors for $\mathfrak{H}$.
\begin{enumerate}[{\rm (1)}]
\item There exists a unique simple ideal $\mathfrak{a}$ of $V_1$ such that $\mathfrak{s}\subset\mathfrak{a}$.
\item Let $Y_{m,k'}$ be the type of $\mathfrak{a}$ in (1).
If a long root of $\mathfrak{s}$ is also a long root of $\mathfrak{a}$, then $X_n$ is contained in $Y_m$ as a root system, and $k=k'$.
Otherwise, $X_n$ is contained in the root system consisting of short roots of $Y_m$.
In particular, $X=A, D$, and $k'=k/2$ (resp. $k'=k/3$) if $Y=B,C,F$ (resp. $Y=G$).
\end{enumerate}
\end{proposition}
\begin{proof} By the assumption, $\mathfrak{s}$ is spanned by root vectors of $V_1$. Let $\mathfrak{a}$ be the ideal generated by $\mathfrak{s}$. Then $\mathfrak{a}$ is a simple ideal by the simplicity of $\mathfrak{s}$.
The uniqueness of $\mathfrak{a}$ follows from the uniqueness of the decomposition of $V_1$ into the direct sum of simple ideals and the center.

Let $(\cdot|\cdot)_{\mathfrak{s}}$ and $(\cdot|\cdot)_{\mathfrak{a}}$ be the normalized Killing forms on $\mathfrak{s}$ and on $\mathfrak{a}$ so that the norm of any long root is $2$, respectively.
By the simplicity of $\mathfrak{s}$, there exists non-zero $\xi\in\C$ such that
\begin{equation}
\xi(\cdot|\cdot)_\mathfrak{s}=(\cdot|\cdot)_\mathfrak{a}\qquad {\rm on}\ \mathfrak{s}.\label{Eq:xi}
\end{equation}
Let $\alpha$ be a root of $\mathfrak{s}$.
Let $E_\alpha$ and $E_{-\alpha}$ be root vectors of $\mathfrak{s}$ associated to $\alpha$ and $-\alpha$, respectively.
By the assumption, both vectors are also root vectors of $\mathfrak{a}$.
Then
\begin{equation}
(E_\alpha)_{(0)}(E_{-\alpha})=(E_\alpha|E_{-\alpha})_\mathfrak{s}\alpha=(E_\alpha|E_{-\alpha})_\mathfrak{a}\alpha',\label{Eq:xi2}
\end{equation}
where $\alpha'$ is the root of $\mathfrak{a}$ corresponding to $E_\alpha$.
Combining \eqref{Eq:xi} and \eqref{Eq:xi2}, we obtain $\alpha=\xi\alpha'$.

Assume that $\alpha$ is a long root of both $\mathfrak{s}$ and $\mathfrak{a}$.
Then $\xi=1$ by $(\alpha|\alpha)_{\mathfrak{s}}=(\alpha|\alpha)_{\mathfrak{a}}=2$.
Thus any long root of $\mathfrak{s}$ is also a long root of $\mathfrak{a}$.
Clearly, the restriction of the normalized invariant form $\langle\cdot|\cdot\rangle$ on $V$ to a subVOA is also the normalized invariant form.
Hence, $k=k'$ since $k(\cdot|\cdot)_{\mathfrak{s}}=\langle\cdot|\cdot\rangle=k'(\cdot|\cdot)_{\mathfrak{a}}$ on $\mathfrak{s}$ (see Lemma \ref{Lem:form}).

Assume that $\alpha$ is a long root of $\mathfrak{s}$ and is not a long root of $\mathfrak{a}$.
Then $\xi\neq1$, and $\alpha'=\alpha/\xi$ is a root of $\mathfrak{a}$.
It follows from $(\alpha/\xi|\alpha/\xi)_{\mathfrak{a}}=(1/\xi)(\alpha|\alpha)_{\mathfrak{s}}=2/\xi\neq 2$ that $\alpha/\xi$ is a short root of $\mathfrak{a}$.
Since $\mathfrak{a}$ is simple, it contains at most two different norms of roots.
Hence the norms of roots of $\mathfrak{s}$ are the same, and $X$ is contained in the root system consisting of short roots of $Y$.
Note that $X=A,D$ and that $\xi=2$ (resp. $\xi=3$) if $Y=B,C,F$ (resp. $Y=G$).
By $$k(\alpha|\alpha)_{\mathfrak{s}}=\langle\alpha|\alpha\rangle=k'(\alpha|\alpha)_{\mathfrak{a}}=k'\xi(\alpha|\alpha)_{\mathfrak{s}},$$
we have $k'=k/\xi$.
\end{proof}

\section{Holomorphic VOA of central charge $24$ with Lie algebra $D_{7,3}A_{3,1}G_{2,1}$}

In this section, applying the $\Z_2$-orbifold construction to a holomorphic VOA of central charge $24$ with Lie algebra $E_{6,3}G_{2,1}^3$ and certain inner automorphism, we obtain a holomorphic VOA of central charge $24$ with Lie algebra $D_{7,3}A_{3,1}G_{2,1}$.

\subsection{Simple affine VOA of type $E_{6,3}$}
Let $\alpha_1,\dots,\alpha_6$ be simple roots of type $E_6$ such that $(\alpha_i|\alpha_j)=-\delta_{j-i,1}+2\delta_{i,j}$, $(\alpha_p|\alpha_1)=-\delta_{p,3}+2\delta_{p,1}$, $(\alpha_p|\alpha_2)=-\delta_{p,4}+2\delta_{p,2}$ for $3\le i\le j\le 6$, $1\le p\le 6$.
Let $\{\Lambda_i\mid 1\le i\le 6\}$ be the set of the fundamental weights  with respect to $\{\alpha_i\mid 1\le i\le 6\}$.
Let $L_\Fg(3,0)$ be the simple affine VOA associated with the simple Lie algebra $\Fg$ of type $E_6$ at level $3$.
There exist exactly $20$ (non-isomorphic) irreducible $L_\Fg(3,0)$-modules $L_\Fg(3,\lambda)$ with highest  weight $\lambda$, which are summarized in Table \ref{E63-module}.

\begin{table}[bht]
\caption{Irreducible $L_\Fg(3,0)$-modules: Case $E_6$} \label{E63-module}
\begin{tabular}{|c|c|c|c|c|c|}
\hline
Highest weight & $0$ & $\Lambda_1$, $\Lambda_6$& $\Lambda_2$&$\Lambda_3$, $\Lambda_5$&$\Lambda_4$\\ \hline
lowest $L(0)$-weight & $0$ & $26/45$&$4/5$& $10/9$&$8/5$ \\ \hline\hline
Highest weight &$\Lambda_1+\Lambda_2$, & $\Lambda_1+\Lambda_3$,  &$\Lambda_1+\Lambda_5$,  & $\Lambda_1+\Lambda_6$ &$2\Lambda_1+\Lambda_6$\\
& $\Lambda_2+\Lambda_6$&$\Lambda_5+\Lambda_6$&$\Lambda_3+\Lambda_6$&&$\Lambda_1+2\Lambda_6$\\ \hline
lowest $L(0)$-weight &$13/9$&$9/5$& $16/9$&$6/5$ &$86/45$\\ \hline\hline
Highest weight & $2\Lambda_1$,$2\Lambda_6$ & $3\Lambda_1$,$3\Lambda_6$&&&\\ \hline
lowest $L(0)$-weight &$56/45$ &$2$&&&\\ \hline
\end{tabular}
\end{table}

One can easily verify the following lemma, which will be used later.

\begin{lemma}\label{Lem:weightsE6}
For every $\alpha\in\Pi(\theta,E_6)$, i.e. root $\alpha$ of $E_6$, we have $((\Lambda_1-\Lambda_6)/2|\alpha)\ge-1/2$.
\end{lemma}

\subsection{Simple affine VOAs of type $G_{2,1}$ and $G_{2,2}$}\label{Sec:G2}
Let $\alpha_1$ and $\alpha_2$ be simple roots of type $G_2$ such that $(\alpha_1|\alpha_1)=2/3$, $(\alpha_2|\alpha_2)=2$ and $(\alpha_1|\alpha_2)=-1$.
Let $\Lambda_1$ and $\Lambda_2$ be the fundamental weights with respect to $\alpha_1$ and $\alpha_2$, respectively.
Let $L_\Fg(k,0)$ be the simple affine VOA associated with the simple Lie algebra $\Fg$ of type $G_2$ at level $k$.
There exist exactly two (resp. four) (non-isomorphic) irreducible $L_\Fg(1,0)$-modules $L_\Fg(1,\lambda)$ (resp. $L_\Fg(2,0)$-modules $L_\Fg(2,\lambda)$) with highest weight $\lambda$, which are summarized in Tables \ref{G21-module} (resp. Table \ref{G22-module}).

\begin{table}[bht]
\caption{Irreducible $L_\Fg(1,0)$-modules: Case $G_2$} \label{G21-module}
\begin{tabular}{|c|c|c|}
\hline
Highest weight & $0$ & $\Lambda_1$ \\ \hline
lowest $L(0)$-weight & $0$ & $2/5$ \\ \hline
\end{tabular}
\end{table}

\begin{table}[bht]
\caption{Irreducible $L_\Fg(2,0)$-modules: Case $G_2$} \label{G22-module}
\begin{tabular}{|c|c|c|c|c|}
\hline
Highest weight & $0$ & $\Lambda_1$&$\Lambda_2$&$2\Lambda_1$ \\ \hline
lowest $L(0)$-weight & $0$ & $1/3$ &$2/3$&$7/9$\\ \hline
\end{tabular}
\end{table}

One can easily verify the following lemma, which will be used later.

\begin{lemma}\label{Lem:weightsG2} Let $\Lambda=\Lambda_2/2$.
\begin{enumerate}[{\rm (1)}]
\item For every $\alpha\in\Pi(\theta,G_2)$, i.e. root $\alpha$ of $G_2$, we have $(\Lambda|\alpha)\ge-1$.
\item For every $\alpha\in\Pi(\Lambda_1,G_2)$, we have $(\Lambda|\alpha)\ge-1/2$.
\item For every $\alpha\in\Pi(2\Lambda_1,G_2)$, we have $(\Lambda|\alpha)\ge-1$.
\end{enumerate}
\end{lemma}

\subsection{Inner automorphism of a holomorphic VOA with Lie algebra $E_{6,3}G_{2,1}^3$}

Let $V$ be a strongly regular, holomorphic VOA of central charge $24$ with Lie algebra $E_{6,3}G_{2,1}^3$.
Note that such a VOA was constructed in \cite{Mi3,SS}.
Let $V_1=\bigoplus_{i=1}^4\mathfrak{g}_i$ be the decomposition into the direct sum of $4$ simple ideals, where
 the type of $\mathfrak{g}_1$ is $E_{6,3}$, and the types of $\mathfrak{g}_2$, $\mathfrak{g}_3$ and $\mathfrak{g}_4$ are $G_{2,1}$.
Let $\mathfrak{H}$ be a Cartan subalgebra of $V_1$.
Then $\mathfrak{H}\cap\mathfrak{g}_i$ is a Cartan subalgebra of $\mathfrak{g}_i$.
Let $U$ be the subVOA generated by $V_1$.
By Proposition \ref{Prop:posl}, $U\cong L_{\mathfrak{g}_1}(3,0)\otimes L_{\mathfrak{g}_2}(1,0)\otimes L_{\mathfrak{g}_3}(1,0)\otimes L_{\mathfrak{g}_4}(1,0)$.
Let $L(\lambda_1,\lambda_2,\lambda_3,\lambda_4)$ denote the irreducible $U$-module  $L_{\mathfrak{g}_1}(3,\lambda_1)\otimes L_{\mathfrak{g}_2}(1,\lambda_2)\otimes L_{\mathfrak{g}_3}(1,\lambda_3)\otimes L_{\mathfrak{g}_4}(1,\lambda_4)$.

Let $${h}=\frac{1}{2}(\Lambda_1-\Lambda_6,\Lambda_2,\Lambda_2,0)\in\bigoplus_{i=1}^4\mathfrak{h}_i.$$
Note that $\Lambda_2=\theta$ in $\mathfrak{g}_i$ $(i=2,3,4)$, the highest root, and that
\begin{equation}
\langle h|h\rangle=\frac{3}{4}(\Lambda_1-\Lambda_6|\Lambda_1-\Lambda_6)_{|\mathfrak{g}_1}+\frac{1}{4}(\Lambda_2|\Lambda_2)_{|\mathfrak{g}_2}+\frac{1}{4}(\Lambda_2|\Lambda_2)_{|\mathfrak{g}_3}=2.\label{Eq:h1}
\end{equation}

\begin{lemma}\label{Lem:moduleE63G21^3}
All the highest weights of irreducible $U$-modules $L(\lambda_1,\lambda_2,\lambda_3,\lambda_4)$ with integral $L(0)$-weights are given by Table \ref{E63G21^3-module}.
In particular, for any weight $(\lambda_1,\lambda_2,\lambda_3,\lambda_4)$ in Table \ref{E63G21^3-module}, we have $( h|(\lambda_1,\lambda_2,\lambda_3,\lambda_4))\in\Z/2$, that is, the spectrum of $h_{(0)}$ on a highest weight vector in $L(\lambda_1,\lambda_2,\lambda_3,\lambda_4)$ is half-integral.
\end{lemma}
\begin{proof} An irreducible module $L(\lambda_1,\lambda_2,\lambda_3,\lambda_4)$ has integral $L(0)$-weights if and only if the sum of the lowest $L(0)$-weights of $L_{\mathfrak{g}_1}(3,\lambda_1)$, $L_{\mathfrak{g}_2}(1,\lambda_2)$, $L_{\mathfrak{g}_3}(1,\lambda_3)$ and $L_{\mathfrak{g}_4}(1,\lambda_4)$ is integral.
Hence, this lemma is immediate from Tables \ref{E63-module} and \ref{G21-module}.
\end{proof}

\begin{table}[bht]
\caption{Irreducible modules with integral $L(0)$-weights: Case $E_{6,3}G_{2,1}^3$} \label{E63G21^3-module}
\begin{tabular}{|c|c|c|c|}
\hline
Highest weight & $(0,0,0,0)$ & $(3\Lambda_1,0,0,0)$& $(\Lambda_4,\Lambda_1,0,0)$\\
&&$(3\Lambda_6,0,0,0)$&$(\Lambda_4,0,\Lambda_1,0)$\\
&&&$(\Lambda_4,0,0,\Lambda_1)$\\
 \hline
lowest $L(0)$-weight & $0$ & $2$&$2$ \\ \hline
\hline
Highest weight &$(\Lambda_1+\Lambda_6,\Lambda_1,\Lambda_1,0)$&$(\Lambda_2,\Lambda_1,\Lambda_1,\Lambda_1)$& $(\Lambda_1+\Lambda_3,\Lambda_1,\Lambda_1,\Lambda_1)$ \\
&$(\Lambda_1+\Lambda_6,\Lambda_1,0,\Lambda_1)$&&$(\Lambda_5+\Lambda_6,\Lambda_1,\Lambda_1,\Lambda_1)$\\
&$(\Lambda_1+\Lambda_6,0,\Lambda_1,\Lambda_1)$&&\\
 \hline
lowest $L(0)$-weight & $2$&$2$&$3$ \\ \hline
\end{tabular}
\end{table}

\begin{lemma}\label{Lem:order2} The spectrum of the $0$-th mode $h_{(0)}$ on $V$ is half-integral.
In particular, $\sigma_h$ is an automorphism of $V$ of order $2$, and the $L(0)$-weights of irreducible $\sigma_h$-twisted $V$-module $V^{(h)}$ are half-integral.
\end{lemma}
\begin{proof}
Let $M$ be an irreducible $U$-submodule of $V$ with highest weight $\lambda$.
Clearly $L(0)$-weights of $M$ are integral.
Hence we can apply Lemma \ref{Lem:moduleE63G21^3} to $M$, and obtain $( h| \lambda)\in\Z/2$.
By the definition of $h$, we have $( h|L)\subset\Z/2$, where $L$ is the root lattice of $V_1$.
Hence for any weight $\mu$ of $M$, we have $(h|\mu)\in\Z/2$ since $\mu\in \lambda+L$.
Thus the spectra of $h_{(0)}$ on $M$ are half-integral, and so is on $V$.
The last assertion follows from Lemma \ref{Lem:wtmodule} and \eqref{Eq:h1}.
\end{proof}

\subsection{Identification of the Lie algebra: Case $D_{7,3}A_{3,1}G_{2,1}$}
In this subsection, we identify the Lie algebra structure of $\tilde{V}_1$.

First, we check the assumptions of Section \ref{Sec:Z2} and Theorem \ref{Thm:CFT}.

\begin{proposition}\label{Prop:VOA} The VOA $V$ and the vector $h$ satisfy the assumptions of Section \ref{Sec:Z2} and Theorem \ref{Thm:CFT}.
In particular, $\tilde{V}=V^{\sigma_h}\oplus (V^{(h)})_\Z$ is a strongly regular, holomorphic VOA of central charge $24$.
In addition, $V_1$ and $\tilde{V}_1$ share a common Cartan subalgebra and their Lie ranks are the same.
\end{proposition}
\begin{proof} By \eqref{Eq:h1}, we have $\langle h|h\rangle=2\in\Z$.
By Lemma \ref{Lem:order2} the order of $\sigma_h$ is $2$.
By Proposition \ref{Prop:V1}, $V_1$ is semisimple.
Let us check the remaining assumptions (a)--(d) in Theorem \ref{Thm:CFT}.
The assumption (a) (resp. (b)) follows from Proposition \ref{Prop:V1} (resp. Lemmas \ref{Lem:weightsG2} and \ref{Lem:weightsE6}).
By the definition of $h$ and the level of $\mathfrak{g}_i$, the assumptions (c) and (d) hold.
\end{proof}

Next, let us determine the Lie algebra structure of $({V}^{\sigma_h})_1$.

\begin{proposition}\label{Prop:rootuntwist}
The set of weights of $(V^{\sigma_h})_1$ for $\mathfrak{H}$ is given as follows:
\begin{align*}
&\{(a,0,0,0)\mid a\in (\bigoplus_{k=2}^5\Z\alpha_k\oplus\Z\theta)\cap\Pi(\theta, E_6)\}
\cup\{\pm(0,\Lambda_2,0,0),\pm(0,0,\Lambda_2,0)\}\\
&\cup\{\pm(0,\alpha_1,0,0),\pm(0,0,\alpha_1,0)\}
\cup\{(0,0,0,\alpha)\mid \alpha\in \Pi(\theta, G_2)\}.
\end{align*}
Moreover, the Lie algebra structure of $(V^{\sigma_h})_1$ is $D_{5,3}A_{1,1}^2A_{1,3}^2G_{2,1}U(1)$ and the dimension of $(V^{\sigma_h})_1$ is $72$, where $U(1)$ means a $1$-dimensional abelian Lie algebra.
\end{proposition}
\begin{proof} One can calculate all the weights of the Lie subalgebra $(V^{\sigma_h})_1$ with respect to $\mathfrak{H}$ directly, which determines the Lie algebra structure of $(V^{\sigma_h})_1$.

The weights $\{\pm(0,\alpha_1,0,0)\}$, $\{\pm(0,0,\alpha_1,0)\}$ are short roots of $\mathfrak{g}_2\oplus\mathfrak{g}_3$ and they form a root system of type $A_1^2$, up to scalar.
Hence, by Proposition \ref{Prop:levels} (2), the type of the corresponding ideal in $(V^{\sigma_h})_1$ is $A_{1,3}^2$.
Since the other components contain long roots of $V_1$, one can determine the levels of the corresponding Lie subalgebras by Proposition \ref{Prop:levels} (2).
\end{proof}

Now, we describe some weights of $(V^{(h)})_1$ for $\mathfrak{H}$ and find a root subsystem of $\tilde{V}_1$.

\begin{proposition}\label{Prop:basis1}
\begin{enumerate}[{\rm (1)}]
\item The $L(0)$-weights of the following four vectors in $V^{(h)}$ are $1$:
\begin{equation}
\1,\quad {E_{(0,-\Lambda_2,0,0)}}_{(-1)}\1,\quad {E_{(0,0,-\Lambda_2,0)}}_{(-1)}\1,\quad {E_{(0,-\Lambda_2,0,0)}}_{(-1)}{E_{(0,0,-\Lambda_2,0)}}_{(-1)}\1.\label{wt1}
\end{equation}
Moreover, they are root vectors in $\tilde{V}_1$ for $\mathfrak{H}$ and their roots are given by
\begin{equation}
\begin{split}
&\frac{1}{2}(3\Lambda_1-3\Lambda_6,\Lambda_2,\Lambda_2,0),\quad \frac{1}{2}(3\Lambda_1-3\Lambda_6,-\Lambda_2,\Lambda_2,0),\\
& \frac{1}{2}(3\Lambda_1-3\Lambda_6,\Lambda_2,-\Lambda_2,0),\quad \frac{1}{2}(3\Lambda_1-3\Lambda_6,-\Lambda_2,-\Lambda_2,0),
\end{split}\label{wt2}
\end{equation}
respectively.
\item There exist root vectors in $\tilde{V}_1$ for $\mathfrak{H}$ whose roots are the negatives of the roots given in \eqref{wt2}.
\end{enumerate}
\end{proposition}
\begin{proof}
It follows from $\langle h|h\rangle=2$, $h_{(0)}\1=0$ and \eqref{Eq:Lh} that $\1\in (V^{(h)})_1$.
Note that $(h|\alpha)=-1$ if $\alpha$ is one of roots $(0,-\Lambda_2,0,0), (0,0,-\Lambda_2,0)$.
By Lemma \ref{Lem:L0weights}, the $L(0)$-weights of the four vectors in \eqref{wt1} are $1$ in $(V^{(h)})_1$.
The explicit construction of $L_\Fg(k,0)$ (\cite{FZ}) shows that all vectors in \eqref{wt1} are non-zero.
In addition, it follows from Lemma \ref{Lem:wtTw} that they are weight vectors for $\mathfrak{H}$ with desired weights, which shows (1).

Recall from Theorem \ref{Thm:CFT} that $\tilde{V}_1$ is a semisimple Lie algebra and that $\mathfrak{H}$ is a Cartan subalgebra of $\tilde{V}_1$.
Hence the vectors in \eqref{wt1} are root vectors of $\tilde{V}_1$.
Obviously (2) holds since the negative of a root is also a root in any root system.
\end{proof}

\begin{proposition}\label{Prop:roots} The set
\begin{align*}
\Psi&=\left\{\left.\pm(0,\Lambda_2,0,0),\pm(0,0,\Lambda_2,0),\ \pm\frac{1}{2}(3\Lambda_1-3\Lambda_6,(-1)^{\delta}\Lambda_2,(-1)^{\varepsilon}\Lambda_2,0)\ \right| \delta,\varepsilon\in\{0,1\}\right\}.\\
\end{align*}
consists of roots of $\tilde{V}_1$ for $\mathfrak{H}$ and it forms a root system of type $A_3$.
Moreover, the level of the Lie subalgebra of $\tilde{V}_1$ corresponding to $\Psi$ is $1$.
\end{proposition}
\begin{proof}
By Propositions \ref{Prop:rootuntwist} and \ref{Prop:basis1}, any element of $\Psi$ is a root of $\tilde{V}_1$.
One can see that the rank of $\Psi$ is $3$, $|\Psi|=12$ and $$\{(0,\Lambda_2,0,0),(0,0,\Lambda_2,0),\frac{1}{2}(3\Lambda_1-3\Lambda_6,-\Lambda_2,-\Lambda_2,0)\}$$ is a set of simple roots of type $A_3$.
Since the level of the Lie subalgebra of type $A_1$ corresponding to $\{\pm(0,\Lambda_2,0,0)\}$ is $1$ (see Proposition \ref{Prop:rootuntwist}), the level of the Lie subalgebra corresponding to $\Psi$ is also $1$ (see Proposition \ref{Prop:levels}).
\end{proof}

Finally, we identify the Lie algebra structure of  $\tilde{V}_1$.

\begin{theorem}\label{Thm:M1} Let $V$ be a strongly regular, holomorphic VOA of central charge $24$.
Assume that the Lie algebra structure of $V_1$ is $E_{6,3}G_{2,1}^3$.
Let $V_1=\bigoplus_{i=1}^4\mathfrak{g}_i$ be the decomposition into the direct sum of $4$ simple ideals, where the types of $\mathfrak{g}_1$, $\mathfrak{g}_2$, $\mathfrak{g}_3$ and $\mathfrak{g}_4$ are $E_{6,3}$, $G_{2,1}$, $G_{2,1}$, $G_{2,1}$, respectively.
Let $h$ be the vector in a Cartan subalgebra $\mathfrak{H}$ of $V_1$ given by$${h}=\frac{1}{2}(\Lambda_1-\Lambda_6,\Lambda_2,\Lambda_2,0)\in \bigoplus_{i=1}^4(\mathfrak{H}\cap\mathfrak{g}_i),$$
where $\Lambda_j$ are the fundamental weights.
Then, applying the $\Z_2$-orbifold construction to $V$ and $\sigma_h$, we obtain a strongly regular, holomorphic VOA $\tilde{V}$ of central charge $24$ whose weight $1$ subspace $\tilde{V}_1$ has a Lie algebra structure $D_{7,3}A_{3,1}G_{2,1}$.
\end{theorem}
\begin{proof}
By Proposition \ref{Prop:VOA}, $\tilde{V}$ is a strongly regular, holomorphic VOA of central charge $24$ and $\mathfrak{H}$ is a Cartan subalgebra of $\tilde{V}$.
Note that $\tilde{V}_1$ is semisimple (Proposition \ref{Prop:V1}) and that the rank of $\tilde{V}_1$ is $12$.

By Proposition \ref{Prop:rootuntwist}, the root system of $\tilde{V}_1$ contains a root system of type $G_2$.
By the classification of root systems, it must be an indecomposable component.
Hence $\tilde{V}_1$ contains a simple ideal of the type $G_{2,1}$.
It follows from Proposition \ref{Prop:V1} that $\dim\tilde{V}_1=120$; hence we obtain the ratio $h^\vee/k =4$ by Proposition \ref{Prop:V1}.

Recall from Propositions \ref{Prop:rootuntwist} and \ref{Prop:roots} that $\tilde{V}_1$ contains simple Lie subalgebras of type $D_{5,3}$ and $A_{3,1}$ which are spanned by weight vectors for $\mathfrak{H}$.
By Proposition \ref{Prop:levels} (1), there exists the simple ideal $\mathfrak{a}$ (resp. $\mathfrak{b}$) of $\tilde{V}_1$ at level $k_a$ (resp. $k_b$) containing the Lie subalgebra of type $A_{3,1}$ (resp. $D_{5,3}$).
By Proposition \ref{Prop:levels} (2), $k_a$ (resp. $k_b$) must be $1$ (resp. $3$), and by the ratio $h^\vee/k=4$, the dual Coxeter number of $\mathfrak{a}$ (resp. $\mathfrak{b}$) is $4$ (resp. $12$).
In addition, the root system of $\mathfrak{a}$ (resp. $\mathfrak{b}$) contains $A_3$ (resp. $D_5$) as in Proposition \ref{Prop:levels} (2).
Hence the only possible type of $\mathfrak{a}$ is $A_{3,1}$, and possible types of $\mathfrak{b}$ are $E_{6,3}$ and $D_{7,3}$.
If the type of $\mathfrak{b}$ is $E_{6,3}$, then the remaining Lie rank is $1$, and hence the Lie algebra structure of $\tilde{V}_1$ is $E_6A_3G_2A_1$ but its dimension is $110$, which contradicts $\dim \tilde{V}_1=120$.
Hence the type of $\mathfrak{b}$ is $D_{7,3}$, and we obtain an ideal of type $D_{7,3}A_{3,1}G_{2,1}$.
Comparing the dimensions of this ideal and $\tilde{V}_1$, we complete the proof.
\end{proof}

\section{Holomorphic VOA of central charge 24 with Lie algebra $E_{7,3}A_{5,1}$}
In this section, applying the $\Z_2$-orbifold construction to a holomorphic VOA of central charge $24$ with Lie algebra $D_{7,3}A_{3,1}G_{2,1}$ and certain inner automorphism, we obtain a holomorphic VOA of central charge $24$ with Lie algebra $E_{7,3}A_{5,1}$.

\subsection{Simple affine VOA of type $D_{7,3}$}
Let $\alpha_1,\dots,\alpha_7$ be simple roots of type $D_7$ such that $(\alpha_i|\alpha_j)=-\delta_{|i-j|,1}+2\delta_{i,j}$, $(\alpha_i|\alpha_7)=-\delta_{i,5}$ and $(\alpha_7|\alpha_7)=2$ for $1\le i,j\le 6$.
Let $\{\Lambda_i\mid 1\le i\le 7\}$ be the set of the fundamental weights 
with respect to $\{\alpha_i\mid 1\le i\le 7\}$.
Let $L_\Fg(3,0)$ be the simple affine VOA associated with the simple Lie algebra $\Fg$ of type $D_7$ at level $3$.
There exist exactly $36$ (non-isomorphic) irreducible $L_\Fg(3,0)$-modules $L_\Fg(3,\lambda)$ with highest weight $\lambda$, which are summarized in Table \ref{D73-module}.

\begin{table}[bht]
\caption{Irreducible $L_\Fg(3,0)$-modules: Case $D_7$ ($j\in\{6,7\}$)} \label{D73-module}
\begin{tabular}{|c|c|c|c|c|c|c|}
\hline
Highest weight & $0$ & $\Lambda_1$& $\Lambda_2$&$\Lambda_3$&$\Lambda_4$&$\Lambda_5$ \\ \hline
lowest $L(0)$-weight & $0$ & $13/30$& $4/5$&$11/10$&$4/3$&$3/2$ \\ \hline\hline
Highest weight & $\Lambda_j$& $2\Lambda_1$ & $3\Lambda_1$& $\Lambda_1+\Lambda_2$&$\Lambda_1+\Lambda_3$&$\Lambda_1+\Lambda_4$\\ \hline
lowest $L(0)$-weight &$91/120$& $14/15$ & $3/2$&$13/10$&$8/5$&$11/6$\\ \hline\hline
Highest weight &$\Lambda_1+\Lambda_5$ &$\Lambda_1+\Lambda_j$&  $2\Lambda_1+\Lambda_j$ & $\Lambda_1+2\Lambda_j$& $\Lambda_1+\Lambda_6+\Lambda_7$&$\Lambda_2+\Lambda_j$\\ \hline
lowest $L(0)$-weight &$2$ &$49/40$& $211/120$ & $32/15$&$21/10$&$13/8$\\ \hline\hline
Highest weight &$\Lambda_3+\Lambda_j$&$\Lambda_4+\Lambda_j$ &$\Lambda_5+\Lambda_j$& $2\Lambda_j$&$3\Lambda_j$&$\Lambda_6+\Lambda_7$ \\ \hline
lowest $L(0)$-weight &$47/24$&$89/40$ &$97/40$&$49/30$ &$21/8$&$8/5$\\ \hline\hline
Highest weight &$2\Lambda_6+\Lambda_7$&$\Lambda_6+2\Lambda_7$&&&&\\ \hline
lowest $L(0)$-weight &$307/120$&$307/120$&&&&\\ \hline
\end{tabular}
\end{table}

Later, we will use the following lemma.

\begin{lemma}\label{Lem:weightsD7} Set $\Lambda=(1/2)(\Lambda_6-\Lambda_7)(=(1/4)(\alpha_6-\alpha_7))$.
Let $j\in\{6,7\}$.
\begin{enumerate}[{\rm (1)}]
\item For every $\alpha\in\Pi(\theta,D_7)$, i.e., root $\alpha$ of $D_7$, we have $(\Lambda|\alpha)\ge-1/2$.
\item Let $\lambda\in\{\Lambda_3, \Lambda_5, \Lambda_6+\Lambda_7\}$.
Then for every $\alpha\in\Pi(\lambda,D_7)$, we have $(\Lambda|\alpha)\ge-1/2$.
\item Let $\lambda\in\{\Lambda_1+\Lambda_3, \Lambda_1+\Lambda_5,\Lambda_1+\Lambda_6+\Lambda_7\}$.
Then for every $\alpha\in\Pi(\lambda,D_7)$, we have $(\Lambda|\alpha)\ge-1$.
\item For every $\alpha\in\Pi(3\Lambda_1,D_7)$, we have $(\Lambda|\alpha)\ge-3/2$.
\item Let $\lambda\in\{\Lambda_1+\Lambda_j, \Lambda_4+\Lambda_j,3\Lambda_j,\Lambda_2+\Lambda_j\}$.
Then for every $\alpha\in\Pi(\lambda,D_7)$, we have $(\Lambda|\alpha)\ge-3/4$.
\end{enumerate}
\end{lemma}
\begin{proof} One can verify this lemma by using explicit descriptions of the weights of irreducible $\mathfrak{g}$-modules\footnote{We use a computer algebra package LiE. \text{http://www-math.univ-poitiers.fr/\~{}maavl/LiE/}}
(cf.\ \cite{Kac}).
\end{proof}

\subsection{Simple affine VOA of type $A_{3,1}$}
Let $\alpha_1,\alpha_2,\alpha_3$ be simple roots of type $A_3$ such that $(\alpha_i|\alpha_j)=-\delta_{|i-j|,1}+2\delta_{i,j}$ for $1\le i,j\le 3$.
Let $\{\Lambda_i\mid 1\le i\le 3\}$ be the set of the fundamental weights with respect to $\{\alpha_i\mid 1\le i\le 3\}$.
Let $L_\Fg(1,0)$ be the simple affine VOA associated with the simple Lie algebra $\Fg$ of type $A_3$ at level $1$.
There exist exactly $4$ (non-isomorphic) irreducible $L_\Fg(1,0)$-modules $L_\Fg(1,\lambda)$ with highest weight $\lambda$, which are summarized in Table \ref{A31-module}.

\begin{table}[bht]
\caption{Irreducible $L_\Fg(1,0)$-modules: Case $A_{3}$} \label{A31-module}
\begin{tabular}{|c|c|c|c|c|}
\hline
Highest weight & $0$ & $\Lambda_1$& $\Lambda_2$&$\Lambda_3$ \\ \hline
lowest $L(0)$-weight & $0$ & $3/8$& $1/2$&$3/8$ \\ \hline
\end{tabular}
\end{table}

One can easily verify the following lemma, which will be used later.

\begin{lemma}\label{Lem:weightsA3}
\begin{enumerate}[{\rm (1)}]
\item For every $\alpha\in \Pi(\theta,A_3)$, i.e., root $\alpha$ of $A_3$, we have $(\Lambda_1|\alpha)\ge-1$.
\item For every $\alpha\in \Pi(\Lambda_1,A_3)$, we have $(\Lambda_1|\alpha)\ge-1/4$.
\item For every $\alpha\in\Pi(\Lambda_2,A_3)$, we have $(\Lambda_1|\alpha)\ge-1/2$.
\item For every $\alpha\in \Pi(\Lambda_3,A_3)$, we have $(\Lambda_1|\alpha)\ge-3/4$.
\end{enumerate}
\end{lemma}

\subsection{Inner automorphism of a holomorphic VOA with Lie algebra $D_{7,3}A_{3,1}G_{2,1}$}

Let $V$ be a strongly regular, holomorphic VOA of central charge $24$ with Lie algebra $D_{7,3}A_{3,1}G_{2,1}$.
Note that such a VOA was constructed in the previous section.
Let $V_1=\bigoplus_{i=1}^3\mathfrak{g}_i$ be the decomposition into the direct sum of $3$ simple ideals, where the types of $\mathfrak{g}_1$, $\mathfrak{g}_2$ and $\mathfrak{g}_3$ are $D_{7,3}$, $A_{3,1}$ and $G_{2,1}$, respectively.
Let $\mathfrak{H}$ be a Cartan subalgebra of $V_1$.
Then $\mathfrak{g}_i\cap\mathfrak{H}$ is a Cartan subalgebra of $\mathfrak{g}_i$.
Let $U$ be the subVOA generated by $V_1$.
By Proposition \ref{Prop:posl}, $U\cong L_{\mathfrak{g}_1}(3,0)\otimes L_{\mathfrak{g}_2}(1,0)\otimes L_{\mathfrak{g}_3}(1,0)$.
Let $L(\lambda_1,\lambda_2,\lambda_3)$ denote the irreducible $U$-module  $L_{\mathfrak{g}_1}(3,\lambda_1)\otimes L_{\mathfrak{g}_2}(1,\lambda_2)\otimes L_{\mathfrak{g}_3}(1,\lambda_3)$.

\begin{table}[bht]
\caption{Irreducible modules with integral $L(0)$-weights: Case $D_{7,3}A_{3,1}G_{2,1}$}\label{A3D7G2-module}
\begin{tabular}{|c|c|c|c|}
\hline
Highest weight & $(0,0,0)$ & $(\Lambda_3,\Lambda_2,\Lambda_1)$&$(\Lambda_1+\Lambda_6+\Lambda_7,\Lambda_2,\Lambda_1)$\\
&& $(\Lambda_5,\Lambda_2,0)$&$(\Lambda_4+\Lambda_j,\Lambda_k,\Lambda_1)$\\
&& $(\Lambda_1+\Lambda_5,0,0)$&$(3\Lambda_j,\Lambda_k,0)$\\
&& $(3\Lambda_1,\Lambda_2,0)$&\\
&&$(\Lambda_1+\Lambda_j,\Lambda_k,\Lambda_1)$&\\
&&$(\Lambda_1+\Lambda_3,0,\Lambda_1)$&\\
&&$(\Lambda_6+\Lambda_7,0,\Lambda_1)$&\\
&&$(\Lambda_2+\Lambda_j,\Lambda_k,0)$&\\
 \hline
lowest $L(0)$-weight & $0$ & $2$&$3$ \\ \hline
\end{tabular}
\end{table}

Let $${h}=\frac{1}{2}(\Lambda_6-\Lambda_7,2\Lambda_1,\Lambda_2)\in\bigoplus_{i=1}^3(\mathfrak{g}_i\cap\mathfrak{H}).$$
Then $$\langle h|h\rangle=\frac{3}{4}(\Lambda_6-\Lambda_7|\Lambda_6-\Lambda_7)_{|\mathfrak{g}_1}+(\Lambda_1|\Lambda_1)_{|\mathfrak{g}_2}+\frac{1}{4}(\Lambda_2|\Lambda_2)_{|\mathfrak{g}_3}=2.$$

\begin{lemma}\label{Lem:moduleA3D7G2}
All the highest weights of irreducible $U$-modules $L(\lambda_1,\lambda_2,\lambda_3)$ with integral $L(0)$-weights are given by Table \ref{A3D7G2-module}, where $k\in\{1,3\}$ and $j\in\{6,7\}$ in the table.
In particular, for any weight $(\lambda_1,\lambda_2,\lambda_3)$ in Table \ref{A3D7G2-module}, we have $( h|(\lambda_1,\lambda_2,\lambda_3))\in\Z/2$, that is, the spectrum of $h_{(0)}$ on a highest weight vector in $L(\lambda_1,\lambda_2,\lambda_3)$ is half-integral.
\end{lemma}
\begin{proof}
This lemma is immediate from Tables \ref{G21-module}, \ref{D73-module} and \ref{A31-module} (cf.\ the proof of Lemma \ref{Lem:moduleE63G21^3}).
\end{proof}

\begin{lemma} The spectrum of the $0$-th mode $h_{(0)}$ on $V$ is half-integral.
In particular, $\sigma_h$ is an automorphism of $V$ of order $2$, and the irreducible $\sigma_h$-twisted $V$-module $V^{(h)}$ has half-integral $L(0)$-weights.
\end{lemma}
\begin{proof} One can prove this lemma by the exactly the same way as in Lemma \ref{Lem:order2} if we use Lemma \ref{Lem:moduleA3D7G2} instead of Lemma \ref{Lem:moduleE63G21^3}.
\end{proof}

\subsection{Identification of the Lie algebra: Case $E_{7,3}A_{5,1}$}
In this subsection, we identify the Lie algebra structure of $\tilde{V}_1$.

First, we determine the Lie algebra structure of $(V^{\sigma_h})_1$.

\begin{proposition}\label{Prop:fixedptD7}
The set of all the weights of $(V^{\sigma_h})_1$ for $\mathfrak{H}$ is given as follows:
\begin{align*}
&\{(\alpha,0,0)\mid \alpha\in (\bigoplus_{k=1}^5\Z\alpha_k\oplus\Z\theta)\cap\Pi(\theta,D_7)\}\cup \{(0,\alpha,0)\mid \alpha\in\Pi(\theta, A_3)\}
\\
&\cup\{\pm(0,0,\alpha_1),\pm(0,0,\Lambda_2)\}.
\end{align*}
Moreover, the Lie algebra structure of $(V^{\sigma_h})_1$ is $D_{6,3}A_{3,1}A_{1,1}A_{1,3}U(1)$ and  $\dim(V^{\sigma_h})_1=88$.
\end{proposition}
\begin{proof} We can find all the weights of $(V^{\sigma_h})_1$ for $\mathfrak{H}\subset(V^{\sigma_h})_1$ from the definition of $h$, and we know that the Lie algebra structure of $(V^{\sigma_h})_1$ is $D_6A_3A_1^2U(1)$.

By the exactly the same arguments as in the proof of Proposition \ref{Prop:rootuntwist}, we can determine the levels of the simple ideals of $(V^{\sigma_h})_1$.
\end{proof}

Next, we determine the lowest $L(0)$-weight of $V^{(h)}$.

\begin{proposition}\label{Prop:low1} The lowest $L(0)$-weight of the irreducible $\sigma_h$-twisted $V$-module $V^{(h)}$ is $1$.
In particular, $(V^{(h)})_{1/2}=0$.
\end{proposition}
\begin{proof} By $\langle h|h\rangle=2$, we have $\1\in (V^{(h)})_1$ by \eqref{Eq:Lh}.
Let $M\cong L(\lambda_1,\lambda_2,\lambda_3)$ be an irreducible $U$-submodule of $V$.
Let $\ell$ and $\ell^{(h)}$ be the lowest $L(0)$-weights of $M$ and of $M^{(h)}$, respectively.
It suffices to show that $\ell^{(h)}\ge1$.
By \eqref{Eq:twisttop} and $\langle h|h\rangle=2$, we have
\begin{align*}
\ell^{(h)}&=\sum_{i=1}^3\left(\frac{(\lambda_i+2\rho|\lambda_i)}{2(k_i+h^\vee)}\right)+\sum_{i=1}^3\left(\min\{(h_i|\lambda_i)\mid \lambda\in\Pi(\lambda_i,X_i)\}\right)+\sum_{i=1}^3\left(\frac{k_i(h_i|h_i)}{2}\right)\notag\\
&=\ell+\sum_{i=1}^3\left(\min\{(h_i|\lambda_i)\mid \lambda\in\Pi(\lambda_i,X_i)\}\right)+1,
\end{align*}
where $X_1=D_7$, $X_2=A_3$ and $X_3=G_2$.
By Lemma \ref{Lem:moduleA3D7G2}, $(\lambda_1,\lambda_2,\lambda_3)$ is one of Table \ref{A3D7G2-module}, and one can see that $\ell^{(h)}\ge1$ by
 Lemmas \ref{Lem:weightsG2},  \ref{Lem:weightsD7} and \ref{Lem:weightsA3}.
\end{proof}

Finally, we identify the Lie algebra structure of $\tilde{V}_1$.

\begin{theorem}\label{Thm:M2} Let $V$ be a strongly regular, holomorphic VOA of central charge $24$.
Assume that the Lie algebra structure of $V_1$ is $D_{7,3}A_{3,1}G_{2,1}$.
Let $V_1=\bigoplus_{i=1}^3\mathfrak{g}_i$ be the decomposition into the direct sum of $3$ simple ideals, where the types of $\mathfrak{g}_1$, $\mathfrak{g}_2$ and $\mathfrak{g}_3$ are $D_{7,3}$, $A_{3,1}$ and $G_{2,1}$, respectively.
Let $h$ be the vector in a Cartan subalgebra $\mathfrak{H}$ given by $${h}=\frac{1}{2}(\Lambda_6-\Lambda_7,2\Lambda_1,\Lambda_2)\in\bigoplus_{i=1}^3(\mathfrak{H}\cap\mathfrak{g}_i),$$
where $\Lambda_i$ are the fundamental weights.
Then, applying the $\Z_2$-orbifold construction to $V$ and $\sigma_h$, we obtain a strongly regular, holomorphic VOA $\tilde{V}$ of central charge $24$ whose weight $1$ subspace $\tilde{V}_1$ has a Lie algebra structure $E_{7,3}A_{5,1}$.
\end{theorem}
\begin{proof} By the exactly the same way as in Proposition \ref{Prop:VOA},
we can apply the $\Z_2$-orbifold construction to $V$ and $\sigma_h$, and we obtain a strongly regular, holomorphic VOA $\tilde{V}$ of central charge $24$.
Notice that $\tilde{V}_1$ is a semisimple Lie algebra of rank $12$.
By Proposition \ref{Prop:fixedptD7}, we have $\dim (V^{\sigma_h})_1=88$, and
by Proposition \ref{Prop:low1}, we have $(V^{(h)})_{1/2}=0$.
By Theorem \ref{Thm:Dimformula} (2), $$\dim\tilde{V}_1=3\times \dim (V^{\sigma_h})_1-\dim V_1+24=168;$$
hence we obtain the ratio $h^\vee/k=6$ by Proposition \ref{Prop:V1}.
By Proposition \ref{Prop:fixedptD7}. $\tilde{V}_1$ contains simple Lie subalgebras of type $D_{6,3}$ and $A_{3,1}$ which are spanned by weight vectors for $\mathfrak{H}$.

By Proposition \ref{Prop:levels} (1), there exists a simple ideal $\mathfrak{a}$ (resp. $\mathfrak{b}$) of $\tilde{V}_1$ at level $k_a$ (resp. $k_b$) containing the Lie subalgebra of type $D_{6,3}$ (resp. $A_{3,1}$).
By Proposition \ref{Prop:levels} (2), $k_a$ (resp. $k_b$) must be $3$ (resp. $1$), and by the ratio $h^\vee/k=6$, the dual Coxeter number of $\mathfrak{a}$ (resp. $\mathfrak{b}$) is $18$ (resp. $6$).
In addition, the root system of $\mathfrak{a}$ (resp. $\mathfrak{b}$) contains $D_6$ (reps. $A_3$) as in Proposition  \ref{Prop:levels} (2).
Hence the possible types of $\mathfrak{a}$ (resp. $\mathfrak{b}$) are $E_{7,3}$ and $D_{10,3}$ (resp. $A_{5,1}$ and $D_{4,1}$).
Since $\dim \tilde{V}_1=168$ and the dimension of a Lie algebra of type $D_{10}$ is $190$, the type of $\mathfrak{a}$ is $E_{7,3}$.
If the type of $\mathfrak{a}$ is $D_{4,1}$, then the remaining rank is $1$, and hence the type of $\tilde{V}_1$ must be $A_1D_4E_7$ but its dimension is $164$, which is a contradiction.
Thus the type of $\mathfrak{b}$ is $A_{5,1}$.
Therefore we obtain an ideal of $\tilde{V}_1$ of type $E_{7,3}A_{5,1}$.
Comparing the dimensions of this ideal and $\tilde{V}_1$, we complete the proof.
\end{proof}

\section{Holomorphic VOA of central charge 24 with Lie algebra $A_{8,3}A_{2,1}^2$}
In this section, applying the $\Z_2$-orbifold construction to a holomorphic VOA of central charge $24$ with Lie algebra $E_{7,3}A_{5,1}$ and certain inner automorphism, we obtain a holomorphic VOA of central charge $24$ with Lie algebra $A_{8,3}A_{2,1}^2$.

\subsection{Simple affine VOA of type $E_{7,3}$}
Let $\alpha_1,\dots,\alpha_7$ be simple roots of type $E_7$ such that $(\alpha_i|\alpha_j)=-\delta_{j-i,1}$ for $3\le i<j\le 7$, and $(\alpha_p|\alpha_p)=2$, $(\alpha_p|\alpha_1)=-\delta_{p,3}$ and $(\alpha_p|\alpha_2)=-\delta_{p,4}$ for $1\le p\le 7$.
Let $\{\Lambda_i\mid 1\le i\le 7\}$ be the set of the fundamental weights with respect to $\{\alpha_i\mid 1\le i\le 7\}$.
Let $L_\Fg(3,0)$ be the simple affine VOA associated with the simple Lie algebra $\Fg$ of type $E_7$ at level $3$.
There exist exactly $12$ (non-isomorphic) irreducible $L_\Fg(3,0)$-modules $L_\Fg(3,\lambda)$ with highest weight $\lambda$, which are summarized in Table \ref{E73-module}.

\begin{table}[bht]
\caption{Irreducible $L_\Fg(3,0)$-modules: Case $E_7$} \label{E73-module}
\begin{tabular}{|c|c|c|c|c|c|c|}
\hline
Highest weight & $0$ & $\Lambda_1$& $\Lambda_2$&$\Lambda_3$&$\Lambda_5$&$\Lambda_6$ \\ \hline
lowest $L(0)$-weight & $0$ & $6/7$&$5/4$&$12/7$&$55/28$&$4/3$ \\ \hline\hline
Highest weight &  $\Lambda_7$&$\Lambda_1+\Lambda_7$ & $\Lambda_2+\Lambda_7$& $\Lambda_6+\Lambda_7$&$2\Lambda_7$&$3\Lambda_7$ \\ \hline
lowest $L(0)$-weight &$19/28$  &$19/12$&$2$&$59/28$&$10/7$&$9/4$ \\ \hline
\end{tabular}
\end{table}

One can easily verify the following lemma, which will be used later.

\begin{lemma}\label{Lem:weightsE7} For every $\alpha\in\Pi(\theta,E_7)$, i.e., root $\alpha$ of $E_7$, we have $(\Lambda_2/2|\alpha)\ge-1$.
\end{lemma}

\subsection{Simple affine VOA of type $A_{5,1}$}
Let $\alpha_1,\alpha_2,\dots,\alpha_5$ be simple roots of type $A_5$ such that $(\alpha_i|\alpha_j)=-\delta_{|i-j|,1}+2\delta_{i,j}$.
Let $\{\Lambda_i\mid 1\le i\le 5\}$ be the set of the fundamental weights with respect to $\{\alpha_i\mid 1\le i\le 5\}$.
Let $L_\Fg(1,0)$ be the simple affine VOA associated with the simple Lie algebra $\Fg$ of type $A_5$ at level $1$.
There exist exactly $6$ (non-isomorphic) irreducible $L_\Fg(1,0)$-modules $L_\Fg(1,\lambda)$ with highest weight $\lambda$, which are summarized in Table \ref{A51-module}.

\begin{table}[bht]
\caption{Irreducible $L_\Fg(1,0)$-modules: Case $A_5$} \label{A51-module}
\begin{tabular}{|c|c|c|c|c|c|c|}
\hline
Highest weight & $0$ & $\Lambda_1$& $\Lambda_2$&$\Lambda_3$&$\Lambda_4$&$\Lambda_5$ \\ \hline
lowest $L(0)$-weight & $0$ & $5/12$& $2/3$&$3/4$&$2/3$&$5/12$ \\ \hline
\end{tabular}
\end{table}

One can easily verify the following lemma, which will be used later.

\begin{lemma}\label{Lem:weightsA5} For every $\alpha\in\Pi(\theta,A_5)$, i.e., root $\alpha$ of $A_5$, $(\Lambda_3/2|\alpha)\ge-1/2$.
\end{lemma}

\subsection{Inner automorphism of a holomorphic VOA with Lie algebra $E_{7,3}A_{5,1}$}

Let $V$ be a strongly regular, holomorphic VOA of central charge $24$ with Lie algebra $E_{7,3}A_{5,1}$.
Note that such a VOA was constructed in the previous section.
Let $V_1=\bigoplus_{i=1}^2\mathfrak{g}_i$ be the decomposition into the direct sum of $2$ simple ideals, where the types of $\mathfrak{g}_1$ and $\mathfrak{g}_2$ are $E_{7,3}$ and $A_{5,1}$, respectively.
Let $\mathfrak{H}$ be a Cartan subalgebra of $V_1$.
Then $\mathfrak{g}_i\cap\mathfrak{H}$ is a Cartan subalgebra of $\mathfrak{g}_i$.
Let $U$ be the subVOA generated by $V_1$.
By Proposition \ref{Prop:posl}, $U\cong L_{\mathfrak{g}_1}(3,0)\otimes L_{\mathfrak{g}_2}(1,0)$.
Let $L(\lambda_1,\lambda_2)$ denote the irreducible $U$-module isomorphic to $L_{\mathfrak{g}_1}(3,\lambda_1)\otimes L_{\mathfrak{g}_2}(1,\lambda_2)$.

\begin{table}[bht]
\caption{Irreducible modules with integral $L(0)$-weights: Case $E_{7,3}A_{5,1}$} \label{E7A5-module}
\begin{tabular}{|c|c|c|c|}
\hline
Highest weight & $(0,0)$ & $(\Lambda_2,\Lambda_3)$&$(3\Lambda_7,\Lambda_3)$\\
&& $(\Lambda_6,\Lambda_2)$&\\
&& $(\Lambda_6,\Lambda_4)$&\\
&& $(\Lambda_1+\Lambda_7,\Lambda_1)$&\\
&& $(\Lambda_1+\Lambda_7,\Lambda_5)$&\\
&& $(\Lambda_2+\Lambda_7,0)$&\\
 \hline
lowest $L(0)$-weight & $0$ & $2$&$3$ \\ \hline
\end{tabular}
\end{table}

Let $${h}=\frac{1}{2}(\Lambda_2,\Lambda_3)\in\bigoplus_{i=1}^2(\mathfrak{g}_i\cap\mathfrak{H}).$$
Then $$\langle h|h\rangle=\frac{3}{4}(\Lambda_2|\Lambda_2)_{|\mathfrak{g}_1}+\frac{1}{4}(\Lambda_3|\Lambda_3)_{|\mathfrak{g}_2}=3.$$

\begin{lemma}\label{Lem:E7A5-1}
All the highest weights of irreducible $U$-modules $L(\lambda_1,\lambda_2)$ with integral $L(0)$-weights are given by Table \ref{E7A5-module}.
In particular, for any weight $(\lambda_1,\lambda_2)$ in Table \ref{E7A5-module}, we have $(h|(\lambda_1,\lambda_2))\in\Z/2$, that is, the spectrum of $h_{(0)}$ on a highest weight vector in $L(\lambda_1,\lambda_2)$ is half-integral.
\end{lemma}
\begin{proof} This lemma is immediate from Tables \ref{E73-module} and \ref{A51-module} (cf.\ the proof of Lemma \ref{Lem:moduleE63G21^3}).
\end{proof}

\begin{lemma}\label{Lem:E7A5half} The spectrum of the $0$-th mode $h_{(0)}$ on $V$ is half-integral.
In particular, $\sigma_h$ is an automorphism of $V$ of order $2$, and the irreducible $\sigma_h$-twisted $V$-module $V^{(h)}$ has half-integral $L(0)$-weights.
\end{lemma}

\begin{proof} One can prove this lemma by the exactly the same way as in Lemma \ref{Lem:order2} if we use Lemma \ref{Lem:E7A5-1} instead of Lemma \ref{Lem:moduleE63G21^3}.
\end{proof}

\subsection{Identification of the Lie algebra: Case $A_{8,3}A_{2,1}^2$}
In this subsection, we identify the Lie algebra structure of $\tilde{V}_1$.

\begin{proposition}\label{Prop:fixedptE7}
The set of all the weights of $(V^{\sigma_h})_1$ for $\mathfrak{H}$ is given as follows:
\begin{align*}
\{(\alpha,0)\mid \alpha\in \Pi(\theta,E_7),\ (\alpha|\Lambda_2)\in2\Z\}\cup\{(0,\alpha)\mid \alpha\in\Pi(\theta,A_5),\ (\alpha|\Lambda_3)\in2\Z\}.
\end{align*}
Moreover, the Lie algebra structure of $(V^{\sigma_h})_1$ is $A_{7,3}A_{2,1}^2U(1)$ and $\dim(V^{\sigma_h})_1=80$.
\end{proposition}
\begin{proof} We can find all the weights of $(V^{\sigma_h})_1$ for $\mathfrak{H}\subset(V^{\sigma_h})_1$ from the definition of $h$, and we can check that the Lie algebra structure of $(V^{\sigma_h})_1$ is $A_{7}A_{2}^2U(1)$.

Since the type of any simple ideal of $V_1$ is $A_n$, the level of a simple ideal of $(V^{\sigma_h})_1$ is equal to that of the ideal of $V_1$ containing it by Proposition \ref{Prop:levels} (2).
\end{proof}

\begin{theorem}\label{Thm:M3} Let $V$ be a strongly regular, holomorphic VOA of central charge $24$.
Assume that the Lie algebra structure of $V_1$ is $E_{7,3}A_{5,1}$.
Let $V_1=\bigoplus_{i=1}^2\mathfrak{g}_i$ be the decomposition into the direct sum of $2$ simple ideals, where the types of $\mathfrak{g}_1$ and $\mathfrak{g}_2$ are $E_{7,3}$ and $A_{5,1}$, respectively.
Let $h$ be the vector in a Cartan subalgebra $\mathfrak{H}$ of $V_1$ given by
$${h}=\frac{1}{2}(\Lambda_2,\Lambda_3)\in\bigoplus_{i=1}^2(\mathfrak{H}\cap\mathfrak{g}_i),$$
where $\Lambda_i$ are the fundamental weights.
Then applying the $\Z_2$-orbifold construction to $V$ and $\sigma_h$, we obtain a strongly regular, holomorphic VOA $\tilde{V}$ of central charge $24$ whose weight $1$ subspace $\tilde{V}_1$ has a Lie algebra structure $A_{8,3}A_{2,1}^2$.
\end{theorem}
\begin{proof}
By the exactly the same way as in Proposition \ref{Prop:VOA}, we can apply the $\Z_2$-orbifold construction to $V$ and $\sigma_h$, and we obtain a strongly regular, holomorphic VOA $\tilde{V}$ of central charge $24$.
Notice that $\tilde{V}_1$ is a semisimple Lie algebra of rank $12$.
By Proposition \ref{Prop:fixedptE7}, $\dim (V^{\sigma_h})_1=88$.
By Theorem \ref{Thm:Dimformula} (2), we have $$\dim\tilde{V}_1=3\times \dim (V^{\sigma_h})_1-\dim V_1+24\times(1-\dim (V^{(h)})_{1/2})=96-24\times\dim (V^{(h)})_{1/2}.$$
Since $\dim\tilde{V}_1\ge\dim (V^{\sigma_h})_1=88$, we have $\dim (V^{(h)})_{1/2}=0$, and $\dim \tilde{V}_1=96$.
Note that the ratio $h^\vee/k$ is $3$ by Proposition \ref{Prop:V1}.
By Proposition \ref{Prop:fixedptE7}, $\tilde{V}_1$ contains a simple Lie subalgebra of type $A_{7,3}$ which is spanned by weight vectors for $\mathfrak{H}$.

By Proposition \ref{Prop:levels} (1), there exists a simple ideal $\mathfrak{a}$ of $\tilde{V}_1$ at level $k$ containing the Lie subalgebra of type $A_{7,3}$.
By Proposition \ref{Prop:levels} (2), $k$ must be $3$, and by the ratio $h^\vee/k=3$, the dual Coxeter number of $\mathfrak{a}$ is $9$.
In addition, the root system of $\mathfrak{a}$ contains $A_7$ as in Proposition \ref{Prop:levels} (2).
Hence the only possible type of $\mathfrak{a}$ is $A_{8,3}$.
By Proposition \ref{Prop:fixedptE7}, $\tilde{V}_1$ also contains a Lie subalgebra of type $A_{2,1}^2$, which has trivial intersection with $\mathfrak{a}$ since the levels are different and $A_{2,1}^2$ is spanned by weight vectors for $\mathfrak{H}$.
Hence $\tilde{V}_1$ contains a Lie subalgebra of type $A_{8,3}A_{2,1}^2$.
Comparing the dimensions of this subalgebra and $\tilde{V}_1$, we complete this theorem.
\end{proof}

\section{Holomorphic VOA of central charge 24 with Lie algebra $A_{5,6}C_{2,3}A_{1,2}$}
In this section, applying the $\Z_2$-orbifold construction to a holomorphic VOA of central charge $24$ with Lie algebra $C_{5,3}G_{2,2}A_{1,1}$ and certain inner automorphism, we obtain a holomorphic VOA of central charge $24$ with Lie algebra $A_{5,6}C_{2,3}A_{1,2}$.

\subsection{Simple affine VOA of type $C_{5,3}$}
Let $\alpha_1,\alpha_2,\dots,\alpha_5$ be simple roots of type $A_5$ such that $(\alpha_i|\alpha_j)=-\delta_{|i-j|,1}/2+\delta_{i,j}$, $1\le i,j\le 4$, $(\alpha_k|\alpha_5)=-\delta_{k,5}$ and $(\alpha_5|\alpha_5)=2$.
Let $\{\Lambda_i\mid 1\le i\le 5\}$ be the set of the fundamental weights with respect to $\{\alpha_i\mid 1\le i\le 5\}$.
Let $L_\Fg(3,0)$ be the simple affine VOA associated with the simple Lie algebra $\Fg$ of type $C_5$ at level $3$.
There exist exactly $56$ (non-isomorphic) irreducible $L_\Fg(3,0)$-modules $L_\Fg(3,\lambda)$ with highest weight $\lambda$, which are summarized in Table \ref{C53-module}.

\begin{table}[bht]
\caption{Irreducible $L_\Fg(3,0)$-modules: Case $C_{5}$} \label{C53-module}
\begin{tabular}{|c|c|c|c|c|c|c|c|}
\hline
Highest weight & $0$ & $\Lambda_1$& $\Lambda_2$&$\Lambda_3$&$\Lambda_4$&$\Lambda_5$&$\Lambda_1^2$ \\ \hline
lowest $L(0)$-weight & $0$ & $11/36$ &$5/9$&$3/4$&$8/9$&$35/36$&$2/3$ \\ \hline\hline
Highest weight & $\Lambda_1\Lambda_2$&$\Lambda_1\Lambda_3$&$\Lambda_1\Lambda_4$&$\Lambda_1\Lambda_5$&$\Lambda_2^2$&$\Lambda_2\Lambda_3$&$\Lambda_2\Lambda_4$\\ \hline
lowest $L(0)$-weight &$11/12$ &$10/9$&$5/4$&$4/3$&$11/9$&$17/12$&$14/9$ \\ \hline\hline
Highest weight &$\Lambda_2\Lambda_5$&$\Lambda_3^2$& $\Lambda_3\Lambda_4$ &$\Lambda_3\Lambda_5$&$\Lambda_4^2$&$\Lambda_4\Lambda_5$&$\Lambda_5^2$\\ \hline
lowest $L(0)$-weight &$59/36$ &$5/3$ &$65/36$&$17/9$&$2$&$25/12$&$20/9$ \\ \hline\hline
Highest weight &$\Lambda_1^3$ &$\Lambda_1^2\Lambda_2$&$\Lambda_1^2\Lambda_3$&$\Lambda_1^2\Lambda_4$&$\Lambda_1^2\Lambda_5$&$\Lambda_1\Lambda_2^2$&$\Lambda_1\Lambda_2\Lambda_3$\\ \hline
lowest $L(0)$-weight &$13/12$ &$4/3$&$55/36$&$5/3$&$7/4$&$59/36$&$11/6$ \\ \hline\hline
Highest weight &$\Lambda_1\Lambda_2\Lambda_4$ &$\Lambda_1\Lambda_2\Lambda_5$ &$\Lambda_1\Lambda_3^2$&$\Lambda_1\Lambda_3\Lambda_4$&$\Lambda_1\Lambda_3\Lambda_5$&$\Lambda_1\Lambda_4^2$&$\Lambda_1\Lambda_4\Lambda_5$\\ \hline
lowest $L(0)$-weight &$71/36$ &$37/18$ &$25/12$&$20/9$&$83/36$&$29/12$&$5/2$ \\ \hline\hline
Highest weight &$\Lambda_1\Lambda_5^2$&$\Lambda_2^3$&$\Lambda_2^2\Lambda_3$&$\Lambda_2^2\Lambda_4$&$\Lambda_2^2\Lambda_5$&$\Lambda_2\Lambda_3^2$&$\Lambda_2\Lambda_3\Lambda_4$\\ \hline
lowest $L(0)$-weight &$95/36$ &$2$ &$79/36$&$7/3$&$29/12$&$22/9$&$31/12$ \\ \hline\hline
Highest weight  &$\Lambda_2\Lambda_3\Lambda_5$&$\Lambda_2\Lambda_4^2$&$\Lambda_2\Lambda_4\Lambda_5$&$\Lambda_2\Lambda_5^2$&$\Lambda_3^3$&$\Lambda_3^2\Lambda_4$&$\Lambda_3^2\Lambda_5$\\ \hline
lowest $L(0)$-weight &$8/3$ &$25/9$ &$103/36$&$3$&$11/4$&$26/9$&$107/36$ \\ \hline\hline
Highest weight &$\Lambda_3\Lambda_4^2$&$\Lambda_3\Lambda_4\Lambda_5$&$\Lambda_3\Lambda_5^2$&$\Lambda_4^3$&$\Lambda_4^2\Lambda_5$&$\Lambda_4\Lambda_5^2$&$\Lambda_5^3$\\ \hline
lowest $L(0)$-weight &$37/12$ &$19/6$ &$119/36$&$10/3$&$41/12$&$32/9$&$15/4$ \\ \hline
\end{tabular}
\end{table}

One can easily verify the following lemma, which will be used later.

\begin{lemma}\label{Lem:weightsC5}
For every $\alpha\in\Pi(\theta,C_5)$, i.e., root $\alpha$ of $C_5$, $(\Lambda_5/2|\alpha)\ge-1/2$.
\end{lemma}

\subsection{Simple affine VOA of type $A_{1,1}$}\label{Sec:A11}
Let $\alpha_1$ be a simple root of type $A_1$ such that $(\alpha_1|\alpha_1)=2$.
Then the fundamental weight is $\Lambda_1=\alpha_1/2$.
Let $L_\Fg(1,0)$ be the simple affine VOA associated with the simple Lie algebra $\Fg$ of type $A_1$ at level $1$.
There exist exactly $2$ (non-isomorphic) irreducible $L_\Fg(1,0)$-modules $L_\Fg(1,\lambda)$ with highest weight $\lambda$, which are summarized in Table \ref{A11-module}.

\begin{table}[bht]
\caption{Irreducible $L_\Fg(1,0)$-modules: Case $A_{1}$} \label{A11-module}
\begin{tabular}{|c|c|c|}
\hline
Highest weight & $0$ & $\Lambda_1$ \\ \hline
lowest $L(0)$-weight & $0$ & $1/4$ \\ \hline
\end{tabular}
\end{table}

\begin{lemma}\label{Lem:weightsA1}
For every $\alpha\in\Pi(\theta,A_1)$, i.e. root $\alpha$ of $A_1$, we have $(\Lambda_1/2|\alpha)\ge-1/2$.
\end{lemma}

\subsection{Inner automorphism of a holomorphic VOA with Lie algebra $C_{5,3}G_{2,2}A_{1,1}$}

Let $V$ be a strongly regular, holomorphic VOA of central charge $24$ with Lie algebra $C_{5,3}G_{2,2}A_{1,1}$.
In this subsection, we assume the existence of such a VOA, which has not been confirmed yet.
Let $V_1=\bigoplus_{i=1}^3\mathfrak{g}_i$ be the decomposition into the direct sum of $3$ simple ideals, where the types of $\mathfrak{g}_1$, $\mathfrak{g}_2$ and $\mathfrak{g}_3$ are $C_{5,3}$, $G_{2,2}$ and $A_{1,1}$, respectively.
Let $\mathfrak{H}$ be a Cartan subalgebra of $V_1$.
Then for $i=1,2,3$, $\mathfrak{g}_i\cap\mathfrak{H}$ is a Cartan subalgebra of $\mathfrak{g}_i$.
Let $U$ be the subVOA generated by $V_1$.
Note that $U\cong L_{\mathfrak{g}_1}(3,0)\otimes L_{\mathfrak{g}_2}(2,0)\otimes L_{\mathfrak{g}_3}(1,0)$.
Let $L(\lambda_1,\lambda_2,\lambda_3)$ denote the irreducible $U$-module $L_{\mathfrak{g}_1}(3,\lambda_1)\otimes L_{\mathfrak{g}_2}(2,\lambda_2)\otimes L_{\mathfrak{g}_3}(1,\lambda_3)$.

\begin{table}[bht]
\caption{Irreducible modules with lowest $L(0)$-weights in $\Z_{\ge2}$: Case $C_{5,3}G_{2,2}A_{1,1}$} \label{C5G2A1-module}
\begin{tabular}{|c|c|c|c|}
\hline
Highest weight &  $(2\Lambda_4,0,0)$&$(\Lambda_2+2\Lambda_5,0,0)$&$(3\Lambda_4,\Lambda_2,0)$\\
& $(3\Lambda_2,0,0)$&$(\Lambda_2+\Lambda_3+\Lambda_5,\Lambda_1,0)$&$(3\Lambda_5,0,\Lambda_1)$\\
&$(2\Lambda_3,\Lambda_1,0)$&$(2\Lambda_5,2\Lambda_1,0)$&$(2\Lambda_4+\Lambda_5,\Lambda_1,\Lambda_1)$\\
&$(2\Lambda_1+\Lambda_4,\Lambda_1,0)$&$(\Lambda_1+\Lambda_3+\Lambda_4,2\Lambda_1,0)$&$(2\Lambda_3+\Lambda_5,2\Lambda_1,\Lambda_1)$\\
&$(2\Lambda_2,2\Lambda_1,0)$&$(2\Lambda_2+\Lambda_4,\Lambda_2,0)$&$(\Lambda_3+2\Lambda_4,\Lambda_2,\Lambda_1)$\\
&$(\Lambda_1+\Lambda_5,\Lambda_2,0)$&$(3\Lambda_3,0,\Lambda_1)$&\\
&$(2\Lambda_1+\Lambda_2,\Lambda_2,0)$&$(\Lambda_1+2\Lambda_4,\Lambda_1,\Lambda_1)$&\\
&$(2\Lambda_1+\Lambda_5,0,\Lambda_1)$&$(2\Lambda_2+\Lambda_5,\Lambda_1,\Lambda_1)$&\\
&$(\Lambda_2+\Lambda_3,\Lambda_1,\Lambda_1)$&$(\Lambda_1+\Lambda_2+\Lambda_4,2\Lambda_1,\Lambda_1)$&\\
&$(\Lambda_5,2\Lambda_1,\Lambda_1)$&$(\Lambda_4+\Lambda_5,\Lambda_2,\Lambda_1)$&\\
&$(3\Lambda_1,\Lambda_2,\Lambda_1)$&$(\Lambda_1+2\Lambda_3,\Lambda_2,\Lambda_1)$&\\
 \hline
lowest $L(0)$-weight & $2$&$3$&$4$ \\ \hline
\end{tabular}
\end{table}

Let $${h}=\frac{1}{2}(\Lambda_5,\Lambda_2,\Lambda_1)\in\bigoplus_{i=1}^3(\mathfrak{g}_i\cap\mathfrak{H}).$$
Then $$\langle h|h\rangle=\frac{3}{4}(\Lambda_5|\Lambda_5)_{|\mathfrak{g}_1}+\frac{1}{2}(\Lambda_2|\Lambda_2)_{|\mathfrak{g}_2}+\frac{1}{4}(\Lambda_1|\Lambda_1)_{|\mathfrak{g}_3}=3.$$

\begin{lemma}\label{Lem:C5G2A1-1}
All the highest weights of irreducible $U$-modules $L(\lambda_1,\lambda_2,\lambda_3)$ whose lowest $L(0)$-weights belong to $\Z_{\ge2}$ are given by Table \ref{C5G2A1-module}.
In particular, for any weight $(\lambda_1,\lambda_2,\lambda_3)$ in Table \ref{C5G2A1-module}, we have $(h|(\lambda_1,\lambda_2,\lambda_3))\in\Z/2$, that is, the spectrum of $h_{(0)}$ on a highest weight vector in $L(\lambda_1,\lambda_2,\lambda_3)$ is half-integral.
\end{lemma}
\begin{proof}
This lemma is immediate from Tables \ref{G22-module}, \ref{C53-module} and \ref{A11-module} (cf.\ the proof of Lemma \ref{Lem:moduleE63G21^3}).
\end{proof}

\begin{lemma}\label{Lem:C5G2A1half} The spectrum of the $0$-th mode $h_{(0)}$ on $V$ is half-integral.
In particular, $\sigma_h$ is an automorphism of $V$ of order $2$, and the irreducible $\sigma_h$-twisted $V$-module $V^{(h)}$ has half-integral $L(0)$-weights.
\end{lemma}
\begin{proof} It follows from $V_i=U_i$ ($i=0,1$) that any irreducible $U$-submodule is isomorphic to $L(\lambda_1,\lambda_2,\lambda_3)$ for some $(\lambda_1,\lambda_2,\lambda_3)$ in Table \ref{C5G2A1-module}.
 One can prove this lemma by the exactly the same way as in Lemma \ref{Lem:order2} if we use Lemma \ref{Lem:C5G2A1-1} instead of Lemma \ref{Lem:moduleE63G21^3}.
\end{proof}

\subsection{Identification of the Lie algebra: Case $A_{5,6}C_{2,3}A_{1,2}$}

In this subsection, we identify the Lie algebra structure of $\tilde{V}_1$.

\begin{proposition}\label{Prop:fixedptC5}
The set of all the weights of $(V^{\sigma_h})_1$ for $\mathfrak{H}$ is given as follows:
\begin{align*}
\{(\alpha,0,0)\mid \alpha\in \Pi(\theta,C_5)\cap\bigoplus_{i=1}^4\Z\alpha_i\}\cup\{\pm(0,\alpha_1,0),\ \pm(0,\Lambda_2,0)\}.
\end{align*}
Moreover, the Lie algebra structure of $(V^{\sigma_h})_1$ is $A_{4,6}A_{1,6}A_{1,2}U(1)^2$ and $\dim(V^{\sigma_h})_1=32$.
\end{proposition}
\begin{proof} We can find all the weights of $(V^{\sigma_h})_1$ for $\mathfrak{H}\subset(V^{\sigma_h})_1$ from the definition of $h$, and we can check that the Lie algebra structure of $(V^{\sigma_h})_1$ is $A_{4}A_{1}^2U(1)^2$.

By the same argument as in Proposition \ref{Prop:rootuntwist},
we see that the level of the Lie subalgebra of $\Fg_2$ of type $A_1$ with roots $\{\pm(0,\Lambda_2,0)\}$ (resp. $\{\pm(0,\alpha_1,0)\}$) is $2$ (resp. $6$) and that the level of the Lie algebra of type $A_4$ is $6$.
\end{proof}

\begin{theorem}\label{Thm:M4} Let $V$ be a strongly regular, holomorphic VOA of central charge $24$.
Assume that the Lie algebra structure of $V_1$ is $C_{5,3}G_{2,2}A_{1,1}$.
Let $V_1=\bigoplus_{i=1}^3\mathfrak{g}_i$ be the decomposition into the direct sum of simple ideals, where the types of $\mathfrak{g}_1$, $\mathfrak{g}_2$ and $\mathfrak{g}_3$ are $C_{5,3}$, $G_{2,2}$ and $A_{1,1}$, respectively.
Let $h$ be the vector in a Cartan subalgebra $\mathfrak{H}$ given by
$${h}=\frac{1}{2}(\Lambda_5,\Lambda_2,\Lambda_1)\in\bigoplus_{i=1}^3(\mathfrak{H}\cap\mathfrak{g}_i),$$
where $\Lambda_i$ is the fundamental weight.
Then applying the $\Z_2$-orbifold construction to $V$ and $\sigma_h$, we obtain a strongly regular, holomorphic VOA $\tilde{V}$ of central charge $24$ whose weight $1$ subspace $\tilde{V}_1$ has the Lie algebra structure $A_{5,6}C_{2,3}A_{1,2}$.
\end{theorem}
\begin{proof}
By the exactly the same way as in Proposition \ref{Prop:VOA}, we can apply the $\Z_2$-orbifold construction to $V$ and $\sigma_h$, and we obtain a strongly regular, holomorphic VOA $\tilde{V}$ of central charge $24$.
Notice that $\tilde{V}_1$ is a semisimple Lie algebra of rank $8$.
By Proposition \ref{Prop:fixedptC5}, $\dim (V^{\sigma_h})_1=32$.
By Theorem \ref{Thm:Dimformula} (2), we have $$\dim\tilde{V}_1=3\times\dim (V^{\sigma_h})_1-\dim V_1+24\times(1-\dim (V^{(h)})_{1/2})=48-24\times \dim (V^{(h)})_{1/2}.$$
By $\dim \tilde{V}_1\ge\dim (V^{\sigma_h})_1=32$, we have $\dim (V^{(h)})_{1/2}=0$, and $\dim \tilde{V}_1=48$;
hence the ratio $h^\vee/k$ is $1$ by Proposition \ref{Prop:V1}.
By Proposition \ref{Prop:fixedptC5}, $\tilde{V}_1$ contains simple Lie subalgebras of type $A_{4,6}$ and $A_{1,2}$ which are spanned by weight vectors for $\mathfrak{H}$.

By Propositions \ref{Prop:levels} (1), there exists a simple ideal $\mathfrak{a}$ (resp. $\mathfrak{b}$) of $\tilde{V}_1$ at level $k_a$ (resp. $k_b$) containing the Lie subalgebra of type $A_{4,6}$ (resp. $A_{1,2}$).

By Proposition \ref{Prop:levels} (2), $k_a$ must be $3$ or $6$, and by the ratio $h^\vee/k=1$ the dual Coxeter number of $\mathfrak{a}$ is equal to $k_a$.
There is no indecomposable root system such that it contains $A_{4}$ and its dual Coxeter number is $3$.
Hence $k_a=6$ and the dual Coxeter number is $6$.
The only possible type of $\mathfrak{a}$ is $A_{5,6}$.

By Proposition \ref{Prop:levels} (2), $k_b$ must be $1$ or $2$, and by the ratio $h^\vee/k=1$, the dual Coxeter number of $\mathfrak{b}$ is equal to $k_b$.
There are no indecomposable root system with dual Coxeter number $1$.
Hence $k_b=2$ and the dual Coxeter number of $\mathfrak{b}$ is $2$.
The only possible type of $\mathfrak{b}$ is $A_{1,2}$.

Let $\mathfrak{c}$ be the ideal of $\tilde{V}_1$ such that $\tilde{V}_1=\mathfrak{a}\oplus\mathfrak{b}\oplus\mathfrak{c}$.
Since the type of $\mathfrak{a}\oplus\mathfrak{b}$ is $A_{5,6}A_{1,2}$, we have $\dim \mathfrak{c}=10$ and the rank of $\mathfrak{c}$ is $2$.
By the semisimplicity of $\mathfrak{c}$, the only possible type of $\mathfrak{c}$ is $C_2(=B_2)$.
In addition, by the ratio $h^\vee/k=1$, its level is $3$.
Comparing the dimensions, we have $\tilde{V}_1=\mathfrak{a}\oplus\mathfrak{b}\oplus\mathfrak{c}$, and the type of $\tilde{V}_1$ is $A_{5,6}C_{2,3}A_{1,2}$.
\end{proof}

\section{Holomorphic VOA of central charge $24$ with Lie algebra $D_{6,5}A_{1,1}^2$}
In this section, we will explain how to obtain a holomorphic VOA of central charge $24$ with Lie algebra  $D_{6,5}A_{1,1}^2$ from a holomorphic VOA constructed by applying the $\Z_5$-orbifold construction to the Niemeier lattice VOA $V_{N(A_4^6)}$.

\subsection{Holomorphic VOA of central charge $24$ with Lie algebra $A_{4,5}^2$}

Let $N$ be a Niemeier lattice with root lattice $A_4^6$.
Let $(\cdot|\cdot)$ be the positive-definite symmetric bilinear form of $\Q\otimes_\Z N$, which will be identified with the normalized Killing form $(\cdot|\cdot)$ on a Cartan subalgebra of the weight $1$ space of the lattice VOA $V_N$.
So we use the same notation.
For explicit calculation, we use the standard model for a root lattice of type $A_4$, i.e.,
\[
A_4=\{ (a_1, \dots, a_5)\in \Z^5\mid a_1+\cdots+a_5=0\}.
\]
Let   $\{\alpha_1=(1,-1,0,0,0) , \alpha_2=(0,1,-1,0,0),\alpha_3=(0,0, 1,-1,0),\alpha_4=(0,0,0,1,-1)\}$ be a set of simple roots.  We also use  the glue code given in \cite[Chapter 16]{CS}, which is the $\Z_5$-code generated by the row vectors of
\[
\begin{pmatrix}
1 & 0&1&4&4&1\\
1 & 1&0&1&4&4\\
1 & 4&1&0&1&4\\
1 & 4&4&1&0&1
\end{pmatrix}.
\]
Let $\tau_0$ be the automorphism of $N$ which acts on $A_4^6$ as a $5$-cycle on the last $5$ copies of $A_4$'s. We denote the induced automorphism on $V_N$ by the same symbol $\tau_0$.
For the details of the lattice VOAs $V_N$, see \cite{Bo,FLM}.

Let $\mathfrak{h}= \C\otimes_\Z N$.
We extend the form $(\cdot|\cdot)$ $\C$-bilinearly to $\mathfrak{h}$.
We also extend the automorphism $\tau_0$ $\C$-linearly to $\mathfrak{h}$.
Let $\mathfrak{h}_{(0)}$ be the subspace of fixed-points of $\tau_0$ in $\mathfrak{h}$.
Note that for $r\in\{\pm1,\pm2\}$, $\mathfrak{h}_{(0)}$ is also the subspace of fixed-points of $\tau_0^r$ in $\mathfrak{h}$ since the order of $\tau_0$ is $5$.
Define
\[
M=((1-P_0)\mathfrak{h}) \cap N = \{\alpha\in N\mid ( \alpha|x) =0 \text{ for all } x\in \mathfrak{h}_{(0)}\},
\]
where $P_0$ is the orthogonal projection from $N$ to $\mathfrak{h}_{(0)}$.
Let $V_N[\tau_0^r]$, ($r=\pm 1,\pm 2$), be the unique irreducible $\tau_0^r$-twisted $V_N$-module (\cite{DLM2}).
Such a module was constructed in \cite{DL} (see \cite[Section 2.2]{SS} for a review) explicitly; as a vector space,
$$V_N[\tau_0^r]\cong M(1)[\tau_0^r]\otimes\C[P_0(N)]\otimes T_r,$$
where $M(1)[\tau_0^r]$ is the ``$\tau_0^r$-twisted" free bosonic space and $T_r$ is the unique irreducible module of $\hat{M}/\widehat{(1-\tau_0^r)N}$ with certain condition (see \cite[Propositions 6.1 and 6.2]{Le} and \cite[Remark 4.2]{DL} for details).
It follows from $(1-\tau_0^r)N=M$ that $\dim T_r=1$ for $r\in\{\pm1,\pm2\}$.
By direct calculation, we have
\begin{equation}
P_0(N)=\left\{\frac{1}{5}(5a,b,b,b,b,b)\mid a\in A_4^*, b\in A_4\right\}.\label{Eq:P0}
\end{equation}

For $r\in\{1,2\}$, set $$\delta^1=\frac15(2,1,0,-1,-2),\ \delta^2=\frac15(-1, 2,0,-2,1)\in \Q\otimes_\Z A_4,\quad {f}^r=(\delta^r, 0,0,0,0,0)\in \mathfrak{h}.$$
Then $( {f}^r| {f}^r) =2/5$ for $r=1,2$.
Note that $5\delta^1$ is the sum of all fundamental weights of $A_4$ and $\delta^2$ is a vector in $2\delta^1+A_4$ with minimum norm.
We regard $f^r$ as a vector in $(V_N)_1$ via the canonical injective map $\mathfrak{h}\to\mathfrak{h}(-1)\1\subset (V_N)_1$.

\begin{remark} We regard $N$ as a lattice in $(V_N)_1$ via the injective map above.
Then the bilinear form $(\cdot|\cdot)$ on $\Q\otimes_\Z N$ coincides with the restriction of the normalized Killing form of $(V_N)_1$ to $\Q\otimes_\Z N$.
In addition, since the level of any simple ideal of $(V_N)_1$ is $1$, the restriction of the normalized invariant from $\langle\cdot|\cdot\rangle$ of $V_N$ to $(V_N)_1$ coincides with $(\cdot|\cdot)$ (see Lemma \ref{Lem:form}).
\end{remark}

Set $\sigma_{f^1}= \exp(-2\pi \sqrt{-1} f^1_{(0)})$.
It follows from $f^1\in N/5$ that $\sigma_{f^1}$ is an automorphism of $V_N$ of order $5$.
It also follows from $\tau_0(f^1)=f^1$ that $\sigma_{f^1}$ commutes with $\tau_0$.  Thus, the automorphism $$g= \tau_0\sigma_{f^1}\in \Aut(V_N)$$ has order $5$.
Since $\delta^2\in2\delta^1+A_4$, we have ${f}^2\in 2{f}^1+N$. Hence
$$
(\sigma_{f^1})^2 = \sigma_{f^2} =  \exp(-2\pi \sqrt{-1} f^2_{(0)})
$$
on $V_{N}$. By Proposition \ref{Prop:twist}, we obtain the irreducible $g^{\epsilon r}$-twisted $V_N$-module $V_N[\tau_0^{\epsilon r}]^{(\epsilon f^r)}$ for $\epsilon=\pm 1$ and $r=1,2$.
For convenience,
we fix a non-zero vector $t_{\epsilon r}\in T_{\epsilon r}$.
Then $T_{\epsilon r}=\C t_{\epsilon r}$.
By \eqref{Eq:Lh}, we have
\begin{equation}
L^{(\epsilon f^r)}(0)=L(0)+ \epsilon f^r_{(0)}+\frac{|\epsilon f^r|^2}2id.\label{Eq:Lpmf}
\end{equation}
Note that $|\varepsilon f^r|^2=\langle f^r|f^r\rangle=({f}^r|{f}^r)=2/5$.

For $\epsilon\in\{\pm1\}$ and $r\in\{1,2\}$, we set
\[
\mathcal{S}^{\epsilon r} = \left\{ a+\epsilon{\delta^r}\ \left| \ a\in A_4^*,\ |a+ \epsilon{\delta^r}|^2 =\frac{2}5\right. \right\}.
\]

\begin{lemma}\label{Lem:S}
Set
\begin{align*}
&\beta_1=\frac{(0,-1,-2,2,1)}5,\ \beta_2=\frac{(2,1,0,-1,-2)}5 ,\ \beta_3=\frac{(-1,-2,2,1,0)}5,\\
&\beta_4= \frac{(1,0,-1,-2,2)}5,\ \beta_0=\frac{(-2,2,1,0,-1)}5.
\end{align*}
Then
\begin{equation}
\begin{split}
&\mathcal{S}^{1}=\{\beta_i\mid i=0,1,2,3,4\},\quad \mathcal{S}^{2} =\left\{ \beta_i+\beta_{i+1}\mid i=0, 1,2,3,4\right\},\\
&\mathcal{S}^{-2} =\left \{ \beta_i+\beta_{i+1}+\beta_{i+2} \mid i=0,1,2,3,4\right\},\\
&\mathcal{S}^{-1} =\left \{ \beta_i+\beta_{i+1}+\beta_{i+2}+\beta_{i+3}\mid i=0,1,2,3,4\right\},
\end{split}\label{Eq:S^er}
\end{equation}
where $i+j$ is interpreted as an integer modulo $5$.
In particular, for  $\epsilon\in\{\pm1\}$ and $r\in\{1,2\}$, we have $|\mathcal{S}^{\epsilon r}|=5$.
\end{lemma}
\begin{proof} Let $a\in A_4^*$.
It follows from $|\epsilon\delta^r|^2=2/5$ that ${|a+\epsilon{\delta^r}|^2}=2/5$ if and only if $ (a| \epsilon{\delta^r}) = -\frac{1}2 |a|^2$.  By the Schwarz inequality, we also have
\[
|( a| \epsilon{\delta^r})| \leq \sqrt{\frac{2}5} |a|.
\]
Thus, $ (a| \epsilon{\delta^r}) = -\frac{1}2 |a|^2$ implies $|a|\leq 2\sqrt{\frac{2}5}$ or $|a|^2\leq 8/5$. Therefore, $a$ is a vector with minimum norm in a coset of $A_4^*/A_4$. By direct calculations, it is easy to verify that there exists a unique $a$ in each coset of $A_4^*/A_4$ such that ${|a+\epsilon{\delta^r}|^2}=2/5$.
Indeed, we obtain all vectors in each $\mathcal{S}^{\epsilon r}$ as in \eqref{Eq:S^er}.
Hence we have proved this lemma.
\end{proof}

\begin{lemma}\label{Lem:twistA4} For $\epsilon\in\{\pm1\}$ and $r\in\{1,2\}$, $$\left\{e^{(a,0,0,0,0,0)}\otimes t_{\epsilon r}\ \left|\ a\in A_4^*,\ |a+\epsilon{f}^r|=\frac25\right.\right\}$$ is a basis of $\left(V_N[\tau_0^{\epsilon r}]^{(\epsilon f^r)}\right)_1$.
Moreover,
the dimension of $\left(V_N[\tau_0^{\epsilon r}]^{(\epsilon f^r)}\right)_1$ is $5$.
\end{lemma}

\begin{proof} Let $w\otimes e^x\otimes t_{\epsilon r}\in V_N[\tau_0^{\epsilon r}]^{(\epsilon f^r)}$ ($w\in M(1)[\tau_0]$, $x\in P_0(N)$) be a vector whose $L(0)$-weight is $1$.
By \cite[(6.28)]{DL}, it is straightforward to show that the $L(0)$-weight of $t_{\epsilon r}\in V_N[\tau_0^{\epsilon r}]$ is $4/5$.
Let $\ell$ be the $L(0)$-weight of $w$ in $M(1)[\tau_0]$, which belongs to $\frac{1}{3}\Z_{\ge0}$.
Then by \eqref{Eq:Lpmf},
the $L(0)$-weight of $w\otimes e^x\otimes t_{\epsilon r}$ in the twisted module $V_N[\tau_0^{\epsilon r}]^{(\epsilon f^r)}$ is
\[
\ell+\frac{|x|^2}{2}+\frac{4}5 +\epsilon ({f}^r |x)+\frac{|{f}^r|^2}2=\ell+\frac{|x+\epsilon{f}^r|^2}{2}+\frac{4}5,
\]
which is equal to $1$ by the assumption.
Hence $\ell=0$, and we may assume that $w=1$.
In addition, we obtain
\begin{equation}
|x+\epsilon{f}^r|^2=\frac25.\label{Eq:x1}
\end{equation}

Let $a\in A_4^*$, $b\in A_4$ such that $x=(5a,b,b,b,b,b)/5\in P_0(N)$ (see \eqref{Eq:P0}).
Then
\begin{equation}
|x+\epsilon f^r|^2=|a+\epsilon\delta^r|^2+\frac{|b|^2}5.\label{Eq:x2}
\end{equation}
Let us show that $|a+\epsilon{\delta^r}|^2\geq \frac{2}5$ for any $a\in A_4^*$.  Clearly,
\[
|a+\epsilon{\delta^r}|^2 = |a|^2 + |{\delta^r}|^2 + {2}\epsilon (a|\delta^r).
\]
Since $\delta^r\in \frac15A_4$ and $a\in A_4^*$, we have $|a|^2\in \frac{2}5\Z$ and $(a|\delta^r)\in \frac15\Z$.
Hence  $|a+\epsilon{\delta^r}|^2\in \frac{2}5 \Z_{\geq 0}$ as $|{\delta^r}|^2=\frac{2}5$. Moreover, $\delta^r \notin A_4^*$ and hence $a+\epsilon\delta^r\neq0$.
Thus
\begin{equation}
|a+\epsilon{\delta^r}|^2 \ge \frac{2}5.\label{Eq:a}
\end{equation}
By \eqref{Eq:x1}, \eqref{Eq:x2} and \eqref{Eq:a}, we have $b=0$ and $|a+\epsilon{\delta^r}|^2=2/5$, which proves the former assertion.

The latter assertion follows from $a+\epsilon{\delta^r}\in\mathcal{S}^{\epsilon r}$ and Lemma \ref{Lem:S}.
\end{proof}

Now consider the following $V_N^g$-module:
\[
\tilde{V}_{N,g}=V_N^g\oplus (V_N[\tau_0]^{(f^1)})_\Z\oplus (V_N[\tau_0^{2}]^{(f^2)})_{\Z}
\oplus (V_N[\tau_0^{-2}]^{(-f^2)})_{\Z}\oplus (V_N[\tau_0^{-1}]^{(-f^1)})_{\Z}.
\]

\begin{remark}
It is claimed by M\"oller Sven that $\tilde{V}_{N,g}$ is a strongly regular, holomorphic VOA of central charge $24$ and that it is a simple current extension of $V_N^g$ graded by $\Z_5$.
\end{remark}

\begin{proposition}
Suppose that  $\tilde{V}_{N,g}$ defined as above is a strongly, holomorphic VOA of central charge $24$. Then $\dim(\tilde{V}_{N,g})_1= 48$ and the Lie algebra structure of $(\tilde{V}_{N,g})_1$ is $A_{4,5}^2$.
\end{proposition}

\begin{proof}
Recall that $\mathfrak{h}_{(0)}=\{(\al,\be,\be,\be,\be,\be)\mid \al,\be\in\C\otimes_\Z A_4\}$.
We view $\mathfrak{h}_{(0)}$ as a subspace of $(V_N^g)_1$.
First we note that
\[
\begin{split}
(V_N^g)_1=\mathfrak{h}_{(0)}\oplus\Span_\C\left\{\left. \sum_{r=-2}^2\tau_0^r\left(e^{(0,\al, 0,0,0,0)}\right)\ \right|\ \al\in A_4,\ (\al|\al)=2\right\}.
\end{split}
\]
The corresponding Lie algebra structure on $(V_N^g)_1$ is $A_{4,5} U(1)^4$ and $\mathfrak{h}_{(0)}$ is a Cartan subalgebra of $(V_N^g)_1$.

Recall from \eqref{Eq:V1h} that for any $x=(x_1,x_2,\dots,x_6)\in\mathfrak{h}_{(0)}\subset (V_N^g)_1$, we have
$$x_{(0)}^{(\epsilon f^r)}=x_{(0)}+(x|\epsilon f^r)id$$
on $V_N[\tau_0^{\epsilon r}](\sigma_f^{\epsilon r})$.
Hence for $w\otimes e^{(a,0,0,0,0,0)} \otimes t_{\epsilon r}\in (V_N[\tau_0^{\epsilon r}]^{(\epsilon f^r)})_1$,
\begin{equation}
x_{(0)}^{(\epsilon f^r)}(w\otimes e^{(a,0,0,0,0,0)} \otimes t_{\epsilon r})
= (x_1| a+ \epsilon \delta^r) w\otimes e^{(a,0,0,0,0,0)} \otimes t_{\epsilon r}. \label{Eq:actx}
\end{equation}
Recall also that $x_{(0)}=0$ on $M(1)[\tau_0]$ and on $T_{\epsilon r}$ by the explicit description of vertex operators in \cite{Le,DL} (cf.\ \cite{SS}).

Let $\beta_0,\beta_1,\dots,\beta_4$ be the vectors in $\Q\otimes_\Z A_4$ given in Lemma \ref{Lem:S}.
Notice that $\mathcal{S}^{\epsilon r}$ in Lemma \ref{Lem:S} is the set of all weights of $(V_N[\tau_0^{\epsilon r}]^{(\epsilon f^r)})_1$ for the Cartan subalgebra $\mathfrak{h}$ of $(V_N^g)_1$.
It is easy to see that
\[
( \beta_i| \beta_j) =
\begin{cases}
2/5 &\text{ if }i=j;\\
-1/5 &\text{ if } |i-j|=1;\\ 0 &\text{ otherwise}.
\end{cases}
\]
Then, up to a scaling, $\{\beta_1, \beta_2, \beta_3, \beta_4\}$ is a set of simple roots for a root system of $A_4$.
By the descriptions of $\mathcal{S}^{\epsilon r}$ in Lemma \ref{Lem:S}, we see that
\[
\{ (\al, 0,0,0,0,0)\mid \al\in \C\otimes_\Z A_4\}\oplus\bigoplus_{r\in\{1,2\},\epsilon\in\{\pm1\}}  \Span_\C\{x\otimes e^{(a, 0,0,0,0,0)}\otimes t_{\epsilon r} \mid
a+{\epsilon \delta^r} \in \mathcal{S}^{\epsilon r}\},
\]
forms a Lie subalgebra of type $A_{4}$.
Notice that $\{(0,\be,\be,\be,\be,\be)\mid \be \in \C\otimes_\Z A_4\}$ is the orthogonal complement of $\{ (\al, 0,0,0,0,0)\mid \al\in \C\otimes_{\Z}A_4\}$ in the Cartan subalgebra $\mathfrak{h}_{(0)}$ and that it acts trivially on this Lie subalgebra (see \eqref{Eq:actx}).
Hence this Lie subalgebra is an ideal.

In order to determine the level of this ideal, we modify the invariant form as $(\cdot|\cdot)_0=(\cdot|\cdot)/5$ on the subspace $\{(\al,0,0,0,0,0)\mid\al\in \C\otimes_\Z A_4\}$ of $\mathfrak{h}_{(0)}$.
Then one can see that $\{\tilde{\beta}_i=5\beta_i\mid 1\le i\le 4\}$ is a set of simple roots and that $\langle\tilde{\beta}_i|\tilde{\beta_i}\rangle=5(\tilde{\beta}_i|\tilde{\beta}_i)_0$.
Hence by Lemma \ref{Lem:form}, the level is $5$.
Therefore the Lie algebra structure of $(\tilde{V}_{N,g})_1$ is $A_{4,5}^2$.
\end{proof}

\begin{remark} It was already claimed in \cite{EMS} that the Lie algebra structure of $(\tilde{V}_{N,g})_1$ is $A_{4,5}^2$ by using Schellekens' list.
\end{remark}

\subsection{Inner automorphism of the holomorphic VOA $\tilde{V}_{N, g}$}
In this subsection, we define an inner automorphism of order $2$ on $\tilde{V}_{N, g}$.

Let $\Lambda=\frac{1}5(\al_1+2\al_2+3\al_3+4\al_4) = \frac{1}5( 1,1,1,1,-4)\in A_4^*$ and $\Lambda'= \be_1+2\be_2+3\be_3+4\be_4=(1,-1,0,-1,1)\in A_4$.
Then $(\Lambda', \Lambda,\Lambda, \Lambda, \Lambda, \Lambda)\in N$ since $(0,1,1,1,1,1)$ belongs to the glue code of $N$.
Set $${h}=\frac{1}2 (\Lambda', \Lambda,\Lambda, \Lambda, \Lambda, \Lambda)\in N/2.$$
Then $\langle h|h\rangle=(h|h)=2$.
Since ${h}$ is fixed by $\tau_0$, we have $h\in\mathfrak{h}_{(0)}\subset (V_N^g)_1\subset (\tilde{V}_{N,g})_1$.
Note that $( {h}| \epsilon{f}^{r}) =0$ for $\epsilon\in\{\pm1\}$, $r\in\{1,2\}$. Now let $\sigma_h=\exp(-2\pi\sqrt{-1}h_{(0)}).$
Then $\sigma_h$ also defines an automorphism in $\Aut \tilde{V}_{N, g}.$

\begin{lemma}\label{Lem:wtNg} On $V_N$ and $V_N[\tau_0^{\epsilon r}]^{(\epsilon f^{r})}$ ($\epsilon=\pm1$, $r=1,2$), $\Spec h_{(0)}\subset \Z/2$.
In particular, the order of $\sigma_h$ is $2$ on $\tilde{V}_{N,g}$.
\end{lemma}
\begin{proof} Since ${h}\in N/2$ and $N$ is unimodular, we have $\Spec h_{(0)}\subset \Z/2$ on $V_N$.

Since $({h}| {\epsilon{f}^{r}}) =0$ and the vectors $\epsilon{f}^{r}$ and $h$ belong to $M(1)_1\subset V_N$, we have $\epsilon{f}^{r}_{(n)}h=0$ for $n\ge0$.
Hence $\Delta(\epsilon{f}^{r},z)h=h$, and
for $w\otimes e^x\otimes t_{\epsilon r}\in V_N[\tau_0^{\epsilon r}]^{({\epsilon f^r})}$,
\[
h_{(0)}^{(\epsilon f^r)}(w\otimes e^x\otimes t_{\epsilon r})={({h}|x)}w\otimes e^x\otimes t_{\epsilon r},
\]
where $w\in M(1)[\tau_0^{\epsilon r}]$, $x=(1/5)(5a,b,b,b,b,b)\in P_0(N)$,  $(a\in A_4^*,\ b\in A_4)$ and $t_{\epsilon r}\in T_{\epsilon r}$.
Since $\Lambda'\in A_4$ and $\Lambda\in A_4^*$, we obtain
\[
(h|x)=\left({h}\left|\frac{1}{5}(5a,b,b,b,b,b)\right.\right)=\frac{( \Lambda'|a)}{2}+\frac{(\Lambda|b)}{2}\in\Z/2,
\]
which completes this lemma.
\end{proof}

\subsection{Identification of the Lie algebra: Case $D_{6,5}A_{1,1}^2$}
In this subsection, we identify the Lie algebra structure of the weight $1$ subspace of the holomorphic VOA $\tilde{V}$ which is obtained by applying the $\Z_2$-orbifold construction to $\tilde{V}_{N,g}$ and $\sigma_h$.

\begin{proposition}\label{Prop:fixedptA45} Let $\tilde{V}_{N,g}$  and $\sigma_h$ be defined as above.
Then the Lie algebra structure of $\left((\tilde{V}_{N,g})^{\sigma_h}\right)_1$ is $A_{3,5}^2 U(1)^2$ and $\dim\left((\tilde{V}_{N,g})^{\sigma_h}\right)_1=32$.
\end{proposition}

\begin{proof}
By the definitions of $\alpha_i$, $\beta_i$, $\Lambda$ and $\Lambda'$, it follows immediately that
\[
( \al_i|\Lambda) = \delta_{i,4}\quad \text{ and } \quad (\be_i|\Lambda') = \delta_{i,4} \quad \text{ for } i=1,2,3,4.
\]
Hence, the Lie algebra structure of $\left((\tilde{V}_{N,g})^{\sigma_h}\right)_1$ is $A_3^2 U(1)^2$.
In addition, the level of the Lie subalgebra of type $A_3^2$ is $5$ by Proposition \ref{Prop:levels} (2).
\end{proof}

\begin{lemma}\label{Lem:Ngh1}
The lowest $L(0)$-weight of the (unique) irreducible $\sigma_h$-twisted $\tilde{V}_{N,g}$-module $(\tilde{V}_{N,g})^{(h)}$ is $1$.
\end{lemma}

\begin{proof}
By $\langle h|h\rangle=2$, \eqref{Eq:Lh}, and Lemma \ref{Lem:wtNg}, we know that the $L(0)$-weights of $(\tilde{V}_{N,g})^{(h)}$ are half-integral.
In addition, we have $L^{(h)}(0)\1=\1$.
In order to prove this lemma, it suffices to show that the lowest $L(0)$-weights of both $(V_N)^{(h)}$ and $\left(V_N[\tau_0^{\epsilon r}]^{(\epsilon f^r)}\right)^{(h)}$ are greater than $1/2$
since $$\left(\tilde{V}_{N,g}\right)^{(h)}\subset (V_N)^{(h)}\oplus \bigoplus_{\epsilon\in\{\pm1\},r\in\{1,2\}}\left(V_N[\tau_0^{\epsilon r}]^{(\epsilon f^r)}\right)^{(h)}.$$

\noindent \textbf{Case $(V_N)^{(h)}$.}
For $a_1(-n_1)\cdots a_i(-n_i)\otimes e^\al \in V_N$, $(n_i\in\Z_{>0},\alpha\in N, a_i\in\mathfrak{h})$, we have
\[
\begin{split}
&\ L^{(h)} (0)\left( a_1(-n_1)\cdots a_i(-n_i)\otimes e^\al \right) \\
= &\ (n_1+\cdots+n_i + \frac{1}2|\al+{h}|^2) \left( a_1(-n_1)\cdots a_i(-n_i)\otimes e^\al \right) .
\end{split}
\]
Hence it suffices to show that $|\al+{h}|^2>1$.

Since $h\in N/2$, we have $\alpha+h\in N/2$.
Let $x_i\in A_4^*$ ($1\le i\le 6$) defined by
\[
2(\al+{h})=(x_1, x_2,x_3,x_4,x_5,x_6) \in N.
\]
Since $\frac{\Lambda'}2, \frac{\Lambda}2 \notin A_4^*$ and $\al\in N\subset(A_4^*)^6$, none of the $x_i$'s is zero and hence
\[
4|(\al+{h})|^2=\sum_{i=1}^6x_i^2\geq \frac{4}5 \times 6>4
\]
since the minimal norm of $A_4^*$ is $4/5$. Hence, $|\al+{h}|^2>1$ as desired.

\medskip

\noindent \textbf{Case $\left(V_N[\tau_0^{\epsilon r}]^{(\epsilon f^{r})}\right)^{(h)}$.}
On $\left(V_N[\tau_0^{\epsilon r}]^{(\epsilon f^{r})}\right)^{(h)}$,
 we have
\begin{align*}
(L^{(\epsilon f^r)})^{(h)}(0)=& \ L^{(h)}(0)+\epsilon {f^r}^{(h)}_{(0)}+\frac{|\epsilon{f}^r|^2}{2} id\\
=& \ L(0)+(\epsilon f^r+h)_{(0)}+\frac{|h+\epsilon{f}^r|^2}2id.
\end{align*}
Let $w\otimes e^{x} \otimes t_{\epsilon r}\in \left(V_N[\tau_0^{\epsilon r}]^{(\epsilon f^{ r})}\right)^{(h)}$ with $L(0)w=\ell w$ $(\ell\ge0)$.
Then
\begin{align*}
& (L^{(\epsilon f^r)})^{(h)}(0) \left( w\otimes e^{x} \otimes t_{\epsilon r}\right) \\
=& \ \left(\ell+\frac{(x|x)}2 +\frac{4}5 +(\epsilon {f}^r+{h}| x)+ \frac{|h+\epsilon{f}^r|^2}2\right)
\left( w\otimes e^{x} \otimes t_{\epsilon r}\right)\\
=& \ \left(\ell+\frac{4}5 + \frac{|h+ \epsilon \bar{f}^r+x |^2}2\right)
\left( w\otimes e^{x} \otimes t_{\epsilon r}\right).
\end{align*}
Thus, the lowest $L(0)$-weight of $\left(V_N[\tau_0^{\epsilon r}](\sigma_f^{\epsilon r})\right)^{(h)}$ is greater than or equal to $4/5$, which completes this case.
\end{proof}

\begin{theorem}\label{Thm:M5}
Let $\tilde{V}$ be the strongly regular, holomorphic VOA of central charge $24$ which is obtained by applying the $\Z_2$-orbifold construction to $\tilde{V}_{N,g}$  and $\sigma_h$. Then
the Lie algebra structure of $\tilde{V}_1$ is $D_{6,5} A_{1,1}^2$.
\end{theorem}

\begin{proof} By $\langle h|h\rangle=2$ and $|\sigma_h|=2$, we can apply the $\Z_2$-orbifold construction to $\tilde{V}_{N,g}$ and $\sigma_h$ and obtain a VOA $\tilde{V}$ of central charge $24$ (Proposition \ref{Prop:VOAstr}).
By Lemma \ref{Lem:Ngh1}, $\tilde{V}$ is of CFT-type.
Similarly to Theorem \ref{Thm:CFT}, we can see that $\tilde{V}$ is strongly regular and holomorphic.
By the definition of $h$, the assumption (d) of Theorem \ref{Thm:CFT} holds.
Hence by Theorem \ref{Thm:CFT} (3), $\tilde{V}_1$ is a semisimple Lie algebra of rank $8$.
By Theorem \ref{Thm:Dimformula} (2), Proposition \ref{Prop:fixedptA45} and Lemma \ref{Lem:Ngh1},
\[
\dim \tilde{V}_1=3\times\dim\left((\tilde{V}_{N,g})^{\sigma_h}\right)_1-\dim(\tilde{V}_{N,g})_1+24\times (1-\dim ((\tilde{V}_{N,g})^{(h)})_{1/2})=72;
\]
hence the ratio $h^\vee/k$ is $2$ by Proposition \ref{Prop:V1}.
By Proposition \ref{Prop:fixedptA45}, $\tilde{V}_1$ contains two simple Lie subalgebras of type $A_{3,5}$ which are spanned by weight vectors for $\mathfrak{H}$.

By Proposition \ref{Prop:levels} (1), there exists a simple ideal $\mathfrak{a}$ of $\tilde{V}_1$ at level $k_a$ containing (one of) the Lie subalgebra of type $A_{3,5}$.
By Proposition \ref{Prop:levels} (2), $k_a$ is $5$, and by Proposition \ref{Prop:V1} the dual Coxeter number of $\mathfrak{a}$ is $10$.
Hence the possible types of $\mathfrak{a}$ are $A_{9,5}$ and $D_{6,5}$.
Since $\dim \tilde{V}_1=72$ and the dimension of a simple Lie algebra of type $A_9$ is $99$, the type of $\mathfrak{a}$ is $D_{6,5}$.

Let $\mathfrak{b}$ be the ideal of $\tilde{V}_1$ such that $\tilde{V}_1=\mathfrak{a}\oplus\mathfrak{b}$.
Then $\dim\mathfrak{b}=6$ and the rank of $\mathfrak{b}$ is $2$.
Since $\mathfrak{b}$ is semisimple, the only possible type of $\mathfrak{b}$ is $A_1^2$.
By the ratio $h^\vee/k=2$, the level of $\mathfrak{b}$ is $1$.
Thus the Lie algebra structure of $\tilde{V}_1$ is $D_{6,5}A_{1,1}^2$.
\end{proof}

\paragraph{\bf Acknowledgement.} The authors wish to thank the referee for useful comments and valuable suggestions.


\begin{thebibliography}{99}

\bibitem[Ap90]{Ap}
T.M.\ Apostol, Modular functions and Dirichlet series in number theory,
Second edition, Graduate Texts in Mathematics, {\bf 41}. Springer-Verlag, New York, 1990. x+204 pp.

\bibitem[Bo86]{Bo}
R.E.\ Borcherds, Vertex algebras, Kac-Moody algebras, and the Monster, {\it Proc.\ Nat'l.\ Acad.\ Sci.\ U.S.A.} {\bf 83} (1986), 3068--3071.


\bibitem[CS99]{CS}
J.H.\ Conway and N.J.A.\ Sloane, Sphere packings, lattices and groups, 3rd Edition, Springer, New York, 1999.

\bibitem[DL96]{DL} C.\ Dong and J.\ Lepowsky,
The algebraic structure of relative twisted vertex operators,
{\it J. Pure Appl. Algebra} {\bf 110} (1996), 259--295.

\bibitem[DLM96]{DLM}
C.\ Dong, H.\ Li,  and G.\ Mason, Simple Currents and Extensions of Vertex Operator
Algebras, {\it Comm. Math. Phys.} \textbf{180} (1996), 671-707.

\bibitem[DLM00]{DLM2}
C.\ Dong, H.\ Li, and G.\ Mason,
Modular-invariance of trace functions
in orbifold theory and generalized Moonshine,
{\it Comm. Math. Phys.} {\bf 214} (2000), 1--56.


\bibitem[DM97]{DMq}
C.\ Dong and G.\ Mason, On quantum Galois theory, {\it Duke Math. J.} {\bf 86} (1997), 305--321.

\bibitem[DM04a]{DMb}
C.\ Dong and G.\ Mason, Holomorphic vertex operator algebras of small central charge, {\it Pacific J. Math.} {\bf 213} (2004), 253--266.

\bibitem[DM04b]{DM}
C.\ Dong and G.\ Mason, Rational vertex operator algebras and the effective central charge, {\it Int. Math. Res. Not.} (2004), 2989--3008.

\bibitem[DM06a]{DM06}
C.\ Dong and G.\ Mason, Integrability of $C_2$-cofinite
vertex operator algebras,
{\it Int. Math. Res. Not.} (2006), Art. ID 80468, 15 pp.

\bibitem[DM06b]{DM06b}
C.\ Dong and G.\ Mason, Shifted vertex operator algebras, {\it Math. Proc. Cambridge Philos. Soc.}  \textbf{141} (2006), 67--80.

\bibitem[EMS]{EMS}
J. van Ekeren, S.\ M\"oller and N.\ Scheithauer, private communication.

\bibitem[FHL93]{FHL}
I.B.\ Frenkel, Y.\ Huang and J.\ Lepowsky, On axiomatic approaches to vertex operator algebras and modules,  {\it Mem. Amer. Math. Soc.} {\bf 104} (1993), viii+64 pp.

\bibitem[FLM88]{FLM}
I.\ Frenkel, J.\ Lepowsky and A.\ Meurman, Vertex operator algebras and the Monster, Pure and Appl.\ Math., Vol.134, Academic Press, Boston, 1988.

\bibitem[FZ92]{FZ}
I.\ Frenkel and Y.\ Zhu, Vertex operator algebras associated to representations
of affine and Virasoro algebras, {\it Duke Math. J.} {\bf 66} (1992),
123--168.

\bibitem[Ha10]{Ha}
K.\ Harada, "Moonshine'' of finite groups, EMS Series of Lectures in Mathematics. European Mathematical Society (EMS), Zurich, 2010.

\bibitem[Hu08]{Huang}
Y.\ Huang, Vertex operator algebras and the Verlinde conjecture, {\it Commun. Contemp. Math.} {\bf 10} (2008), 103--154.

\bibitem[Ka90]{Kac}
V.G.\ Kac, Infinite-dimensional Lie algebras, Third edition, Cambridge University Press, Cambridge, 1990.

\bibitem[KM12]{KM}
M.\ Krauel and G.\ Mason, Vertex operator algebras and weak Jacobi forms, {\it Internat. J. Math.} {\bf 23} (2012), 1250024, 10 pp.

\bibitem[La11]{Lam}
C.H.\ Lam, On the constructions of holomorphic vertex operator algebras of
central charge $24$, {\it Comm. Math. Phys.} {\bf 305} (2011), 153--198

\bibitem[LS12]{LS}
C.H.\ Lam and H.\ Shimakura, Quadratic spaces and holomorphic framed vertex operator algebras of central charge 24, {\it Proc. Lond. Math. Soc.} {\bf 104} (2012), 540--576.

\bibitem[LS15]{LS2} C.H.\ Lam and H.\ Shimakura, Classification of holomorphic framed vertex operator algebras of central charge 24, {\it Amer. J. Math.} {\bf 137} (2015), 111--137.

\bibitem[Le85]{Le} J. Lepowsky, Calculus of twisted vertex operators,  {\it Proc.
Natl. Acad. Sci. USA} {\bf 82} (1985), 8295--8299.

\bibitem[Li94]{Li3}
H.\ Li, Symmetric invariant bilinear forms on vertex operator algebras, {\it J. Pure Appl. Algebra}, {\bf 96} (1994), 279--297.

\bibitem[Li96]{Li}
H.\ Li,  Local systems of twisted vertex operators, vertex operator superalgebras and twisted modules, {\it in} Moonshine, the Monster, and related topics, 203--236, {\it Contemp. Math.}, {\bf 193}, Amer. Math. Soc., Providence, RI, 1996.

\bibitem[Li97]{Li2}
H.\ Li, Extension of vertex operator algebras by a self-dual simple module, {\it J. Algebra} {\bf 187} (1997), 236--267.

\bibitem[Ma14]{Ma}
G.\ Mason, Lattice subalgebras of strongly regular vertex operator algebras, {\it in} Conformal Field Theory, Automorphic Forms and Related Topics, 31--53, {\it Contrib. Math. Comput. Sci.} {\bf 8}, Springer, Heidelberg, 2014.

\bibitem[Mi13]{Mi3}
M.\ Miyamoto, A $\Z_3$-orbifold theory of
lattice vertex operator algebra and
$\Z_3$-orbifold constructions,
{\it in} Symmetries, integrable systems and representations, 319--344,
{\it Springer Proc. Math. Stat.} {\bf 40}, Springer, Heidelberg, 2013.

\bibitem[Mi15]{Mi}
M.\ Miyamoto, $C_2$-cofiniteness of cyclic-orbifold models, {\it Comm. Math. Phys.} {\bf 335} (2015),  1279--1286.

\bibitem[Mi]{Mi4}
M.\ Miyamoto, Flatness and Semi-Rigidity of Vertex Operator Algebras, arXiv:1104.4675.

\bibitem[Mo94]{Mo}
P.S.\ Montague, Orbifold constructions and the classification of self-dual $c=24$ conformal field theories, {\it Nuclear Phys.} B {\bf 428} (1994), 233--258.

\bibitem[SS]{SS}
D.\ Sagaki and H.\ Shimakura, Application of a $\mathbb{Z}_{3}$-orbifold construction to the lattice vertex operator algebras associated to Niemeier lattices, to appear in {\it Trans. Amer. Math. Soc.}

\bibitem[Sc93]{Sc93}
A.N.\ Schellekens, Meromorphic $c=24$ conformal field theories, {\it Comm. Math. Phys.} {\bf 153} (1993), 159--185.

\bibitem[Zh96]{Zhu}
Y.\ Zhu, Modular invariance of characters of vertex operator algebras, {\it J. Amer. Math. Soc.} {\bf 9} (1996), 237--302.

\end{thebibliography}
\end{document}